\documentclass{amsart}

\usepackage{amssymb}
\usepackage{amssymb, latexsym}
\usepackage{hyperref}
\usepackage{amsmath}
\usepackage{amscd}
\usepackage{enumerate}
\usepackage{amsfonts}
\usepackage{graphicx}
\usepackage[all]{xypic}
\usepackage{mathrsfs}

\allowdisplaybreaks

\newcommand{\Z}{{\mathbb{Z}}}
\newcommand{\C}{{\mathbb{C}}}
\newcommand{\R}{{\mathbb{R}}}
\newcommand{\N}{{\mathbb{N}}}

\theoremstyle{plain}
\newtheorem{theorem}{Theorem}
\newtheorem{proposition}[theorem]{Proposition}
\newtheorem{lemma}[theorem]{Lemma}
\newtheorem{corollary}[theorem]{Corollary}

\theoremstyle{definition}

\numberwithin{equation}{section}
\numberwithin{theorem}{section}

\numberwithin{equation}{section}

\newcommand{\KKK}{\mathcal{K}}

\newcommand{\LLL}{\mathcal{L}}

\begin{document}


\title[NLS with the combined terms]
{The dynamics of the NLS with the combined terms \\
in five and higher dimensions}

\author[Miao]{Changxing Miao}
\address{\hskip-1.15em Changxing Miao:
\hfill\newline Institute of Applied Physics and Computational
Mathematics, \hfill\newline P. O. Box 8009,\ Beijing,\ China,\
100088,
\newline
Beijing Center of Mathematics and Information Sciences,
Beijing, 100048, P.R.China.
}
\email{miao\_changxing@iapcm.ac.cn}

\author[Xu]{Guixiang Xu}
\address{\hskip-1.15em Guixiang Xu \hfill\newline Institute of
Applied Physics and Computational Mathematics, \hfill\newline P. O.
Box 8009,\ Beijing,\ China,\ 100088, }
\email{xu\_guixiang@iapcm.ac.cn}

\author[Zhao]{Lifeng Zhao}
\address{\hskip-1.15em Lifeng Zhao \hfill\newline
University of Science and Technology of China, \hfill\newline
Hefei,\ China, } \email{zhaolifengustc@yahoo.cn}

\subjclass[2000]{Primary: 35L70, Secondary: 35Q55}

\keywords{Blow up; Dynamics; Nonlinear Schr\"{o}dinger Equation;
Scattering; Threshold Energy.}

\maketitle

{\em Dedicated to Professor Shanzhen Lu on the occasion of his 75 birthday}

\begin{abstract}In this paper, we continue the study in \cite{MiaoWZ:NLS:3d Combined} to show the scattering and blow-up result of the solution
 for the nonlinear Schr\"{o}dinger equation with the energy
below the threshold $m$ in the energy space $H^1(\R^d)$,
\begin{align} iu_t +  \Delta u  =  -|u|^{4/(d-2)}u + |u|^{4/(d-1)}u, \; d\geq 5. \tag{CNLS}
\end{align}  The threshold is given by the
ground state $W$ for the energy-critical NLS: $iu_t +  \Delta u  =
-|u|^{4/(d-2)}u$. Compared with the argument in \cite{MiaoWZ:NLS:3d Combined}, the new ingredient is that we use
the double duhamel formula in \cite{ Kiv:Clay Lecture, TaoVZ:NLS:mass compact} to
lower the regularity of the critical element in $L^{\infty}_tH^1_x$ to $L^{\infty}\dot H^{-\epsilon}_x$ for some $\epsilon>0$ in five and higher dimensions
and obtain the compactness of the critical element in $L^2_x$, which is used to control the spatial center function $x(t)$ of the critical element and furthermore
used to defeat the critical element in the reductive argument.
\end{abstract}


%
%
%
%

\section{Introduction}

We consider the dynamics of the energy solutions for the
nonlinear Schr\"{o}dinger equation (NLS) with the combined
nonlinearities in five and higher dimensions
\begin{equation} \label{NLS}
\left\{ \aligned
    iu_t +  \Delta u  = &\; f_1(u) + f_2(u),   \quad (t,x)\in \R \times \R^d , \\
     u(0)= & \; u_0(x)\in H^1(\R^d).
\endaligned
\right.
\end{equation}
where $u:\R \times \R^d \mapsto \C$, $d\geq 5$ and $f_1(u)=-|u|^{4/(d-2)}u$, $f_2(u)=
|u|^{4/(d-1)}u$.

The equation has the following mass and Hamiltonian
\begin{align*}
M(u)(t)= & \frac12 \int_{\R^d} |u(t,x)|^2 \; dx; \quad  E(u)(t)=
\int_{\R^d} \frac12 |\nabla u(t,x)|^2 \; dx  + F_1(u(t)) + F_2(u(t))
\end{align*}
where
\begin{align*}  F_1(u(t)) = \displaystyle  -\frac{d-2}{2d} \int_{\R^d}
|u(t,x)|^{\frac{2d}{d-2}}\; dx,\;\;  F_2(u(t))= \frac{d-1}{2d+2}  \int_{\R^d} |u(t,x)|^{\frac{2d+2}{d-1}} \;
dx.
\end{align*}
They are conserved for the sufficient smooth solutions of
\eqref{NLS}.

In \cite{TaoVZ:NLS:combined}, Tao, Visan and Zhang made the
comprehensive study of
\begin{align*}
iu_t +  \Delta u  =  |u|^{\frac{4}{d-2}}u + |u|^{\frac{4}{d-1}}u
\end{align*}
in the energy space. They made use of the interaction Morawetz
estimate established in \cite{CKSTT04} and the stability theory for
the scattering solution. Their result is based on the scattering
result of the defocusing, energy-critical NLS in the energy space,
which is established by Bourgain \cite{Bou:NLS:99, Bou:NLS:book} for
the radial case, I-team \cite{CKSTT07}, Ryckman-Visan \cite{RyV05}
and Visan \cite{Vi05} for the nonradial case. Since the classical
interaction Morawetz estimate fails for \eqref{NLS}, Tao, et al.,
leave the scattering and blow-up dichotomy of \eqref{NLS} below the
threshold as an open problem in \cite{TaoVZ:NLS:combined}. For other
results, please refer to \cite{AkaIKN:NLS:combined:blowup,
AkaIKN:NLS:combined:scattering, DuyHR:NLS:GWPS, GinV:85:NLS,
GinV:85:NLKG subcritical, HolRou:NLS:GWPS, HolRou:NLS:BU, Nak:NLKG
low dim:subcrit, Nak:01:NLKG subcritical, NakSch:cubic NLS:Rigidity,
TaoVZ:NLS:combined, Zha:global:NLKG, Zhang:NLS:06}.

For the focusing, energy-critical NLS
\begin{align}\label{NLS:focusing critical}
iu_t +  \Delta u = -|u|^{\frac{4}{d-2}}u.
\end{align}
Kenig and Merle first applied the concentration
compactness in \cite{BahG:NLW:proffile decomp, Ker:NLS:profile decomp, Ker:NLS:compactness}
into the scattering theory of the radial solution of
\eqref{NLS:focusing critical} in \cite{KenM:NLS:GWP} with the energy below that of the
ground state of
\begin{align*}
-\Delta W = |W|^{\frac{4}{d-2}}W.
\end{align*}
Subsequently, Killip and Visan made use of the double Duhamel
argument in \cite{KiTV:NLS:2d, TaoVZ:NLS:mass compact} to removed
the radial assumption in \cite{KiV:en-NLS:high dim}. For the
applications of the concentration compactness in the scattering and
the blow up theory of the NLS, NLW, NLKG and Hartree equations,
please see \cite{Dod:NLS:higher dim, Dod:NLS:two dim, Dod:NLS:one
dim, Dod:NLS:foc, DuyHR:NLS:GWPS, DuyM:NLS:ThresholdSolution,
DuyR:NLS:ThresholdSolution, HolRou:NLS:GWPS, HolRou:NLS:BU,
IbrMN:f:NLKG, KenM:NLW:GWP, KiTV:NLS:2d, KiV:en-NLS:high dim,
KiVZ:NLS:high dim, KriNS:e-critical NLW, KriNakSch:1D KG, LiZh:NLS,
MiaoWX:Har:dynamic, MiXZ:Hartree:De-rad, MiXZ07b, MiXZ:Hartree:low
regul, MiXZ:Hartree:mass critical, MiXZ:NLS:e-crit
Hartree:nonradial}.

In \cite{MiaoWZ:NLS:3d Combined}, we made use of the concentration compactness argument and rigidity argument to
show the dichotomy of the radial solution of \eqref{NLS} in $H^1(\R^3)$ with energy
less than the threshold.
In this paper, we continue this study in five and higher dimensions.

Now for $\varphi\in H^1$, we denote the scaling quantity
$\varphi^{\lambda}_{d, -2}$ by
\begin{align*}
\varphi^{\lambda}_{d, -2} (x)= e^{d\lambda}\varphi(e^{2\lambda}x).
\end{align*}
We denote the scaling derivative of $E$ by $K(\varphi)$
\begin{align}\label{scaling deriv:special}
K(\varphi)=  \LLL E(\varphi)
:=  \dfrac{d}{d \lambda } \Big|_{\lambda =0 } E(
\varphi^{\lambda}_{d, -2})  =  \int_{\R^d} 2 |\nabla
\varphi|^2 - 2 |\varphi|^{\frac{2d}{d-2}} + \frac{2d}{d+1} |\varphi|^{\frac{2d+2}{d-1}} \; dx,
\end{align}
which is connected with the Virial identity, and then plays the
important role in the blow-up and scattering of the solution of
\eqref{NLS}.

The threshold is determined by the following constrained
minimization of the energy $E(\varphi)$
\begin{align}\label{minimization}
m = \inf \{ E(\varphi)\; |\; 0\not=\varphi \in H^1(\R^d), \;
K(\varphi)=0 \}.
\end{align}
Since we consider the $\dot H^1$-critical growth with the
subcritical perturbation, we need the following modified energy
\begin{align*}
E^c(u)= & \int_{\R^d} \left(\frac12 |\nabla u(t,x)|^2 -\frac{d-2}{2d}
|u(t,x)|^{\frac{2d}{d-2}} \right) \; dx.
\end{align*}

As the nonlinearity $|u|^{\frac{4}{d-1}}u$ is the defocusing, $\dot
H^1$-subcritical perturbation, one think that the focusing, $\dot
H^1$-critical term plays the decisive role of the threshold of the
scattering solution of \eqref{NLS} in the energy space. Just as the
3d case in \cite{MiaoWZ:NLS:3d Combined}, the first result is to
characterize the threshold energy $m$ as the following

\begin{proposition}[\cite{MiaoWZ:NLS:3d Combined}, Proposition 1.1]\label{threshold-energy} For $d\geq 5$, there is no minimizer for \eqref{minimization}. But for the threshold energy
$m$, we have
\begin{align*}
m = E^c(W),
\end{align*}
where $W\in   H^1(\R^d)$ is the ground state of the massless equation
\begin{align*}
-\Delta W = |W|^{\frac{4}{d-2}}W.
\end{align*}
\end{proposition}

Main result in this paper is
\begin{theorem}\label{theorem} For $d\geq 5$. Let $u_0\in H^1(\R^d) $ with
\begin{equation}
E(u_0) <  m,
\end{equation}
and $u$ be the solution of \eqref{NLS} and $I$ be its maximal
interval of existence. Then
\begin{enumerate}
\item[\rm (a)] If $K(u_0)\geq 0$, then $I=\R$, and $u$ scatters in both time directions as $t\rightarrow \pm \infty$ in
$H^1$;

\item[\rm (b)] If $K(u_0)<0$ and $xu_0\in L^2$ or $u_0$ is radial , then $u$ blows up both
forward and backward at finite time in $H^1$.
\end{enumerate}
\end{theorem}

By the above result, we conclude that the focusing, $\dot H^1$-critical term make the main contribution to
the determination of the threshold of the scattering solution of \eqref{NLS}. For the case $d=3$, we verify the above result for the radial case in \cite{MiaoWZ:NLS:3d Combined}.
In this paper, we show the scattering result without the radial assumption in five and higher dimensions.
Compared with the argument in \cite{MiaoWZ:NLS:3d Combined}, the new ingredient in five and higher dimensions is that we can use
the double duhamel formula in \cite{ Kiv:Clay Lecture, TaoVZ:NLS:mass compact} to
lower the regularity of the critical element in $L^{\infty}_tH^1_x$ to $L^{\infty}\dot H^{-\epsilon}_x$ for some $\epsilon > 0$
and obtain the compactness of the critical element in $L^2_x$, which is used to control the spatial center function $x(t)$ of the critical element and furthermore
used to defeat the critical element in the reductive argument.

At last, from the assumption in Theorem \ref{theorem}, we know that the solution starts from the following subsets of the energy space,
\begin{align*}
\KKK^{+}=& \Big\{ \varphi \in H^1(\R^3)\; \;\Big|\;\; \varphi\; \text{is radial},\; E(\varphi)<m,\; K(\varphi)\geq 0  \Big\},\\
\KKK^{-}=& \Big\{ \varphi \in H^1(\R^3)\; \;\Big|\;\; \varphi\; \text{is radial},\; E(\varphi)<m,\; K(\varphi)< 0  \Big\}.
\end{align*}
By the similar scaling argument to that in \cite{MiaoWZ:NLS:3d Combined}, we know that $\KKK^{\pm} \not= \emptyset$.

This paper is organized as follows. In Section \ref{S:pre}, we give the basic well-known results, including the linear and nonlinear estimates,
the local well-posedness, the perturbation theory and the monotonicity formula. In Section \ref{S:var}, we show the threshold by the variational
method, which also give the proof of Proposition \ref{threshold-energy} and various variational estimates, which will be used in the proof of Theorem
\ref{theorem}. In Section \ref{S:blow up}, we give the proof of the blow up in Theorem
\ref{theorem} in the radial case. In Section \ref{S:profiledec}, we show the linear and nonlinear profile decompositions of the $H^1$-bound
sequences of solution of \eqref{NLS}. In Section \ref{S:GWP-Scattering}, we make use of the stability theory and compactness argument to show the
global wellposedness and scattering result in Theorem \ref{theorem} in a reductive argument.

%
%
%
%

\section{Preliminaries}\label{S:pre}
In this section, we give some notation and some well-known results.
\subsection{Littlewood-Paley decomposition and Besov space}
Let $\Lambda_0(x) \in \mathcal{S}(\R^d)$ such that its Fourier
transform $\widetilde{\Lambda}_0(\xi) =1$ for $|\xi|\leq 1$ and
$\widetilde{\Lambda}_0(\xi) = 0 $ for $|\xi|\geq 2$. Then we define
$\Lambda_k(x)$ for any $k\in \Z\backslash \{0\}$ and $\Lambda_{(0)}(x)$ by the
Fourier transforms:
\begin{align*}
\widetilde{\Lambda}_k(\xi)=\widetilde{\Lambda}_0(2^{-k}\xi) -
\widetilde{\Lambda}_0(2^{-k+1}\xi), \quad
\widetilde{\Lambda}_{(0)}(\xi)=\widetilde{\Lambda}_0(\xi)-\widetilde{\Lambda}_0(2\xi).
\end{align*}
Let $s\in \R$, $1\leq p, q \leq \infty$. The inhomogeneous Besov space $B^s_{p,q}$ is defined by
\begin{align*}
B^s_{p,q}=\left\{f \; \big|\; f\in \mathcal{S}'(\R^d), \Big\|2^{ks}\big\|\Lambda_k*f\big\|_{L^p_x}\Big\|_{l^q_{k\geq 0}}<\infty\right\},
\end{align*}
where $\mathcal{S}'(\R^d)$ denotes the space of tempered distributions. The homogeneous Besov space $\dot B^s_{p,q}$ can be
defined by
\begin{align*}
\dot B^s_{p,q}=\left\{f \; \Big|\; f\in {\mathcal Z}'(\R^d), \left( \sum_{k\in\Z\backslash \{0\}}  2^{qks}\big\|\Lambda_k*f\big\|^q_{L^p_x}  + \big\|\Lambda_{(0)}*f\big\|_{L^p_x}\Big\|^q\right)^{1/q}<\infty\right\}.
\end{align*}
 ${\mathcal Z}'(\R^d)$ denotes the dual space of ${\mathcal Z}(\R^d)=\{
f\in {\mathcal S}(\R^d); \partial^\alpha\hat f(0)=0; \forall \alpha\in
\mathbb N^d \,\hbox {multi-index}\}$ and can be identified by the quotient
space of ${\mathcal S}'/{\mathcal P}$ with the polynomials space ${\mathcal P}$.

\subsection{Linear estimates and nonlinear estimates}  We say that a pair of exponents $(q,r)$ is Schr\"{o}dinger $\dot H^s$-admissible in $d\geq 5$ if
\begin{align*}
\dfrac2q+\dfrac{d}{r}=\dfrac{d}{2}-s
\end{align*}
and $2\leq q, r \leq \infty$. If $I\times \R^d$ is a space-time
slab, we define the $\dot S^0(I\times \R^d)$ Strichartz norm by
\begin{align*}
\big\| u\big\|_{\dot S^0(I\times \R^d)}:=\sup\big\|u\big\|_{L^q_t
L^r_x(I\times \R^d)}
\end{align*}
where the sup is taken over all $L^2$-admissible pairs $(q,r)$. We define
the $ \dot  S^s(I\times \R^d)$ and $ S^s(I\times \R^d)$ Strichartz norm to be
\begin{align*}
\big\| u\big\|_{\dot S^s(I\times \R^d)}:=\big\| D^s u\big\|_{\dot
S^0(I\times \R^d)}, \quad \big\| u\big\|_{  S^s(I\times \R^d)}:=\big\| \left<\nabla\right>^s u\big\|_{\dot
S^0(I\times \R^d)}..
\end{align*}
We also use $\dot N^0(I\times
\R^d)$ to denote the dual space of $\dot S^0(I\times \R^d)$ and
\begin{align*}
\dot N^k(I\times \R^d):=\{u; D^k u \in \dot N^0(I\times \R^d)\}.
\end{align*}

Before we introduce the linear estimate, we first give some exponents, which will be frequently used in the paper. For
\begin{align*}
S(I) : = & L^{\infty}\left(I; L^2\right) \cap L^2\left(I; L^{2^*}\right), \\
W_1(I) : = & L^{2\frac{d+2}{d-2}} \left(I; L^{2\frac{d+2}{d-2}} \right),  \quad
V_1(I) : =  L^{2\frac{d+2}{d-2}} \left(I; L^{2\frac{d(d+2)}{d^2+4}}\right),\\
W_2(I) : = & L^{2\frac{d+2}{d-1}} \left(I; L^{2\frac{d+2}{d-1}} \right), \quad
V_2(I) : =   L^{2\frac{d+2 }{d-1}} \left(I; L^{2\frac{ d( d+2 )}{d^2+2}} \right).
\end{align*}

By definition and Sobolev's inequality, we have
\begin{lemma} For any $ \dot  S^1$ function $u$ on $I\times \R^d$, we have
\begin{align*}
\big\|\nabla u\big\|_{S(I)}  + \big\|\nabla u \big\|_{V_1(I)}+
\big\|u\big\|_{ W_1(I) }  + \big\|\nabla u\big\|_{V_2(I)}  + \big\||\nabla|^{1/2} u\big\|_{ W_2(I) } \lesssim \big\|u \big\|_{\dot  S^1}.
\end{align*}
\end{lemma}

Now we state the standard Strichartz estimate.

\begin{lemma}[\cite{Caz:NLS:book, KeT98, tao:book}]\label{Strichartz }
Let $I$ be a compact time interval, $k\in [0,1]$, and let $u: I\times
\R^d\rightarrow \C$ be an $\dot S^k$ solution to the forced
Schr\"{o}dinger equation
\begin{align*}
iu_t + \Delta u =F
\end{align*}
for a function $F$. Then we have
\begin{align*}
\big\|u\big\|_{\dot S^k(I\times \R^d)} \lesssim
\big\|u(t_0)\big\|_{\dot H^k(\R^d)} + \big\|F\big\|_{\dot
N^k(I\times \R^d)},
\end{align*}
for any time $t_0\in I$.
\end{lemma}

For $d\geq 5$, let $\displaystyle s_{\alpha}: = \frac{d}{2} - \frac{2}{\alpha -1} = 1-\frac2d$, then $\displaystyle \alpha = \frac{d^2+2d+4}{d^2-2d+4}$.
Let $(\gamma, \rho)$ be the $\dot H^{s_{\alpha}}$-admissible pair such that
\begin{align*}
\rho = \frac{\alpha + 1 + 2^*}{2},  \quad \frac2{\gamma} = d\left(\frac12 - \frac{1}{\rho}\right)-s_{\alpha}.
\end{align*}

\begin{lemma}[\cite{Fos:NLS:exotic}]\label{strichartz:exotic} For $d\geq 5,$  and any $F\in L^2_t \left(I; \dot
B^{\frac2d}_{\frac{2d^2}{d^2+4}, 2} \right) $, we have
\begin{align*}
\left\|\int^t_0e^{i(t-s)\Delta} F(s)\; ds \right\|_{L^{\gamma}\left(I; \dot B^{\frac2d}_{\rho, 2} \right)}
\lesssim  & \big\| F\big\|_{L^2_t \left(I; \dot
B^{\frac2d}_{\frac{2d^2}{d^2+4}, 2} \right)}.
\end{align*}
\end{lemma}

Let $(q_i, r_i)$, $i=1,\ldots, 6$ be the exponentials such that
\begin{align*}
\frac{1}{2} = &\; \frac{4/(d-2)}{q_1} + \frac{1}{\gamma} =\; \frac{1}{q_2} + \frac{1}{q_3} + \frac{1}{\gamma}
= \;  \frac{4/(d-1)}{q_4} + \frac{1}{\gamma}  = \;   \frac{1}{q_5} + \frac{1}{q_6} + \frac{1}{\gamma},\\
\frac{d^2+4}{2d^2} = & \; \frac{4/(d-2)}{r_1} + \frac{1}{\rho} = \; \frac{1}{r_2} + \frac{1}{r_3} + \frac{d^2 - 2\rho }{d^2 \rho}
 = \;  \frac{4/(d-1)}{r_4} + \frac{1}{\rho} = \; \frac{1}{r_5} + \frac{1}{r_6} + \frac{d^2 - 2\rho }{d^2 \rho},
\end{align*}
where
\begin{align*}
\frac1 {q_2} = \frac1  {r_2} =   \frac{d-2}{2(d+2)} \times \left(\frac{4}{d-2} -\frac4d\right), \quad
\frac1 {q_5} = \frac1  {r_5} =   \frac{d-1}{2(d+2)} \times \left(\frac{4}{d-1} -\frac4d\right),
\end{align*}
then
\begin{enumerate}
\item $A_1=(q_1, r_1)$ is $\dot H^{1 }$-admissible pair; $W_1=\left(\Big(\frac{4}{d-2}-\frac{4}{d}\Big)q_2, \big( \frac{4}{d-2}-\frac{4}{d}\Big)r_2 \right)$ is the
diagonal $\dot H^{1 }$-admissible pair;  $B_1=(\frac4d q_3, \frac4d r_3)$ is $\dot H^{1/2}$-admissible pair.

\item $A_2=(q_4, r_4)$ is $\dot H^{1/2}$-admissible pair; $W_2=\left(\Big(\frac{4}{d-1}-\frac{4}{d}\Big)q_5, \big( \frac{4}{d-1}-\frac{4}{d}\Big)r_5 \right)$ is the
diagonal $\dot H^{1/2}$-admissible pair;  $B_2=(\frac4d q_6, \frac4d r_6)$ is $L^2$-admissible pair.

\item $ES = \left(\gamma, \rho\right)$, $ES^* = \left(2,  \frac{2d^2}{d^2+4} \right)$.
\end{enumerate}

 \begin{figure}
\centering
\includegraphics[width=0.5\textwidth]{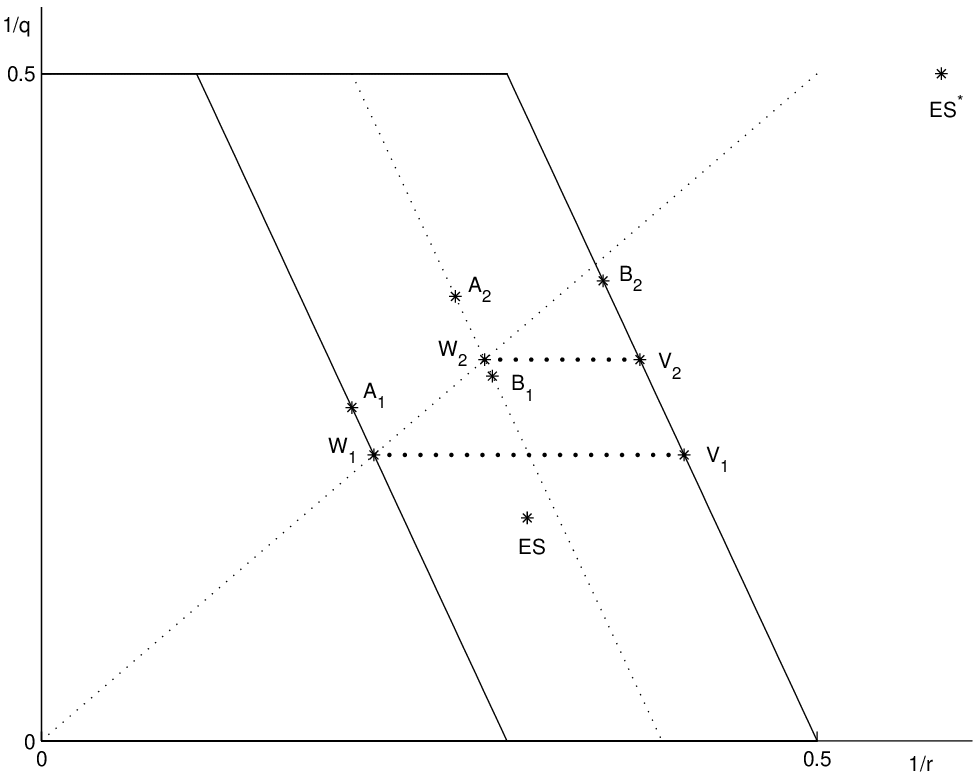}
\caption[]{ Admissible pairs: $A_i, B_i, V_i, W_i$, $i=1,2$, and $ES$, $ES^*$ }
\end{figure}

\begin{lemma}[\cite{AkaIKN:NLS:combined:scattering, Vi05}]
For $d\geq 5$, assume that $h_1$ and $h_2$ are H\"{o}lder continuous functions of order $\frac4{d-2}$ and $\frac{4}{d-1}$, respectively. Let $I$
be an interval, then we have
\begin{align*}
\big\|h_1(v)w\big\|_{ L^2_t \left(I;
B^{\frac2d}_{\frac{2d^2}{d^2+4}, 2} \right) } \lesssim & \big\|  v\big\|^{\frac4{d-2} }_{L^{q_1}_tL^{r_1}_x} \big\|w\big\| _{L^{\gamma}\left(I;   B^{\frac2d}_{\rho, 2} \right)}
  \\
  & +  \big\|  v \big\|^{\frac4{d-2}-\frac4d}_{L^{ \left(\frac{4}{d-2}-\frac{4}{d}\right)q_2}_tL^{ \left(\frac{4}{d-2}-\frac{4}{d}\right)r_2}_x}
 \big\||\nabla|^{1/2} v\big\|^{\frac{4}{d}}_{L^{\frac4d q_3}_t L^{\frac4d r_3}_x}
 \big\|w\big\| _{L^{\gamma}\left(I;  B^{\frac2d}_{\rho, 2} \right)}, \\
\big\|h_2(v)w\big\|_{ L^2_t \left(I;
B^{\frac2d}_{\frac{2d^2}{d^2+4}, 2} \right) } \lesssim & \big\|  v\big\|^{\frac4{d-1} }_{L^{q_4}_tL^{r_4}_x} \big\|w\big\| _{L^{\gamma}\left(I;  B^{\frac2d}_{\rho, 2} \right)}
  \\
  & +  \big\|  v \big\|^{\frac4{d-1}-\frac4d}_{L^{ \left(\frac{4}{d-1}-\frac{4}{d}\right)q_5}_tL^{ \left(\frac{4}{d-1}-\frac{4}{d}\right)r_5}_x}
 \big\||\nabla|^{1/2} v\big\|^{\frac{4}{d}}_{L^{\frac4d q_6}_t L^{\frac4d r_6}_x}
 \big\|w\big\| _{L^{\gamma}\left(I;  B^{\frac2d}_{\rho, 2} \right)} .\end{align*}
\end{lemma}

\subsection{Local wellposedness and perturbation theory}\label{S:LWP and Pertubation}

Let us denote $ST(I) $ by $ W_1(I) \cap W_2(I).$ By the analogue analysis as those in \cite{AkaIKN:NLS:combined:scattering, MiaoWZ:NLS:3d Combined}, we have
\begin{theorem}[Local wellposedness,  \cite{AkaIKN:NLS:combined:scattering, MiaoWZ:NLS:3d Combined, TaoVZ:NLS:combined}]\label{lwp}
Let $u_0\in H^1$, then for every $T>0$, there exists $\eta=\eta(T)$
such that if
\begin{align*}
\big\|\left<\nabla\right> e^{it\Delta}u_0\big\|_{V_2([-T, T])}\leq \eta,
\end{align*}
then \eqref{NLS} admits a unique strong $H^1_x$-solution $u$ defined
on $[-T, T]$. Let $(-T_{min}, T_{max})$ be the maximal time interval
on which $u$ is well-defined. Then, $u\in S^1(I\times \R^d)$ for
every compact time interval $I\subset (-T_{min}, T_{max})$ and the
following properties hold:
\begin{enumerate}
\item If $T_{max}<\infty$, then
\begin{align*}
\big\|u\big\|_{ ST ((0, T_{max})\times \R^d)}=\infty.
\end{align*}
Similarly, if $T_{min}<\infty$, then
\begin{align*}
\big\|u\big\|_{ST  ((-T_{min}, 0)\times \R^d)}=\infty.
\end{align*}
\item The solution $u$ depends continuously on the initial data
$u_0$ in the following sense: The functions $T_{min}$ and $T_{max}$
are lower semicontinuous from $  H^1_x  $ to $(0, +\infty]$. Moreover,
if $u^{(m)}_0 \rightarrow u_0$ in $  H^1_x  $ and $u^{(m)}$ is the
maximal solution to \eqref{NLS} with initial data $u^{(m)}_0$, then
$u^{(m)}\rightarrow u$ in $ST(I\times \R^d)$ and every compact subinterval $I\subset (-T_{min}, T_{max})$.
\end{enumerate}
\end{theorem}

\begin{proposition}[\cite{AkaIKN:NLS:combined:scattering, MiaoWZ:NLS:3d Combined}]\label{stability}
Let $I$ be a compact time interval and let $w$ be an approximate
solution to \eqref{NLS} on $I\times \R^d$ in the sense that
\begin{align*}
i\partial_t w + \Delta w= - |w|^{\frac{4}{d-2}} w+|w|^{\frac4{d-1}} w +e
\end{align*}
for some suitable small function $e$. Assume that for some constants $L, E_0>0$, we
have
\begin{align*}
\big\|w\big\|_{ST(I)} \leq  L, \quad  \big\|w(t_0)\big\|_{
H^1_x(  \R^d)}  \leq  E_0.
\end{align*}
for some $t_0 \in I$. Let $u(t_0)$ close
to $w(t_0)$ in the sense that for some $E'>0$, we have
\begin{align*}
\big\|u(t_0)-w(t_0)\big\|_{ H^1_x}\leq E'.
\end{align*}
 Assume also that for some $\varepsilon$, we have
\begin{align} \label{small:gamma0}
  \left\| \left<\nabla\right> e^{i(t-t_0)\Delta}\big(u(t_0)-w(t_0)\big) \right\|_{V_2(I)} & \leq   \varepsilon,\quad \big\| \left<\nabla\right> e\big\|_{L^{2\frac{d+2}{d+4}}(I)} \leq   \varepsilon,
\end{align}
where $  0<\varepsilon \leq
\varepsilon_0=\varepsilon_0(E_0, E', L) $ is a small constant.
Then there exists a solution $u$ to \eqref{NLS} on $I\times \R^d$
with initial data $u(t_0)$ at time $t=t_0$ satisfying
\begin{align*}
\big\| \left<\nabla\right>\left(u-w\right) \big\|_{V_1(I)} \leq& C(E_0, E', L) \;
\varepsilon, \quad \text{and}\quad  \big\|\left<\nabla\right>u\big\|_{  S  (I)} \leq
C(E_0, E', L).
\end{align*}
\end{proposition}

\subsection{Monotonicity formula}
\begin{lemma}[\cite{Gla:NLS:blowup}]\label{L:virial}
Let $\phi\in C^{\infty}_0(\R^3)$, and $u$ be the solution of \eqref{NLS}. Then we have
\begin{align*}
\partial_t \int_{\R^d} \phi (x) \big| u(t,x)\big|^2\; dx
= &   -2 \Im \int_{\R^d}\nabla \phi \cdot \nabla \bar{u} \; u \; dx\\
\partial^2_t \int_{\R^d} \phi (x) \big| u(t,x)\big|^2\; dx
= &    4 \int_{\R^d}\phi_{ij}(x)  u_i(t,x)\bar{u}_j(t,x)\;  dx
- \int_{\R^d} \Delta^2 \phi \big| u(t,x)\big|^2\; dx\\
& - \frac{4}{d}\int_{\R^d} \Delta \phi \big| u(t,x)\big|^{2^*}\; dx+\frac{4}{d+1}\int_{\R^3} \Delta \phi \big| u(t,x)\big|^{\frac{2d+2}{d-1}}\; dx.
\end{align*}
\end{lemma}

\section{Variational characterization}\label{S:var} In this section, we show the threshold energy $m$
(Proposition \ref{threshold-energy}) by the variational method,
and various estimates for the solutions of \eqref{NLS} with the energy below
the threshold. The argument is the analogue as the case  $d=3$ in \cite{MiaoWZ:NLS:3d Combined}.

Let us denote the quadratic and nonlinear parts of $K$ by $K^Q$ and
$K^N$, that is,
\begin{align*}
K(\varphi)=K^Q(\varphi) + K^N(\varphi),
\end{align*}
where $K^Q(\varphi)= \displaystyle 2 \; \int_{\R^d} |\nabla \varphi|^2 \;
dx,$ and $K^N (\varphi)= \displaystyle \int_{\R^d} \left(-2 |  \varphi|^{\frac{2d}{d-2}} + \frac{2d}{d+1} |\varphi|^{\frac{2d+2}{d-1}} \right)\;
dx$.

\begin{lemma}\label{positive near origin:KQ }
For any $\varphi \in H^1(\R^d)$, we have
\begin{align}\label{asymptotic:KQ}
\lim_{\lambda \rightarrow -\infty}
K^Q(\varphi^{\lambda}_{d,-2}) =0.
\end{align}
\end{lemma}
\begin{proof}
It is obvious by the definition of $K^Q$.
\end{proof}

Now we show the positivity of $K$ near 0 in the energy
space.
\begin{lemma}\label{Postivity:K} For any bounded sequence
$\varphi_n\in H^1(\R^d) \backslash\{0\}$ with
\begin{align*}
\lim_{n\rightarrow +\infty}K^Q(\varphi_n)=0,
\end{align*}
then for large $n$, we have
\begin{align*}
K(\varphi_n)>0.
\end{align*}
\end{lemma}
\begin{proof} By the fact that $K^Q(\varphi_n) \rightarrow 0$, we know that
$\displaystyle
\lim_{n\rightarrow+\infty}\big\|\nabla\varphi_n\big\|^2_{L^2}=0.
$ Then by the Sobolev and Gagliardo-Nirenberg inequalities, we have for large
$n$
\begin{align*}
\big\|\varphi_n\big\|^{2^*}_{L^{2^*}_x} \lesssim  & \big\|\nabla
\varphi_n\big\|^{2^*}_{L^2_x} =  o(\big\|\nabla\varphi_n\big\|^2_{L^2}),\\
\big\|\varphi_n\big\|^{\frac{2d+2}{d-1}}_{L^\frac{2d+2}{d-1}_x}\lesssim &
\big\|\varphi_n\big\|^{\frac{2}{d-1}}_{L^2}\big\|\nabla \varphi_n\big\|^{\frac{2d}{d-1}}_{L^2}   =
o(\big\|\nabla\varphi_n\big\|^2_{L^2}),
\end{align*}
where we use the boundedness of $\big\|\varphi_n\big\|_{L^2}$. Hence for large $n$, we have
 \begin{align*} K(\varphi_n)= &  \int_{\R^d}
\left( 2 |\nabla \varphi_n|^2 -2|\varphi_n|^{2^*} +\frac{2d}{d+1} |\varphi_n|^{\frac{2d+2}{d-1}}
\right)\; dx \thickapprox   \int_{\R^d}  |\nabla \varphi_n|^2\; dx > 0.
\end{align*}
This concludes the proof.
\end{proof}

By the definition of $K$, we denote two real numbers
\begin{align*}
\bar{\mu} = \max\{4,0, \frac{4d}{d-1}\}=\frac{4d}{d-1}, \quad \underline{\mu}=\min\{4,0,
\frac{4d}{d-1}\}=0.
\end{align*}

Next, we show the behavior of the scaling derivative functional $K$ with respect to the scaling $\varphi^\lambda_{d,-2}$.

\begin{lemma}\label{structure:J}
For any $\varphi \in H^1$, we have
\begin{align*}
\left(\bar{\mu}-\LLL\right)E(\varphi) = & \int_{\R^d} \left( \frac{2}{d-1}
\big|\nabla \varphi\big|^2  + \frac{2}{d-1}
\big|\varphi\big|^{2^*} \right)\; dx, \\
\LLL  \left(\bar{\mu}-\LLL\right)E(\varphi) = &
\int_{\R^d}\left(\frac{8}{d-1}\big|\nabla \varphi\big|^2 +
\frac{8d}{(d-1)(d-2)}\big|\varphi\big|^{2^*} \right) \; dx,
\end{align*}
where the scaling derivative $\LLL$ is defined by \eqref{scaling
deriv:special}.
\end{lemma}
\begin{proof} By the definition of $\LLL$,
we have
\begin{align*}
\LLL \big\|\nabla \varphi\big\|^2_{L^2} = 4 \big\|\nabla
\varphi\big\|^2_{L^2},  \quad \LLL \big\|  \varphi\big\|^{2^*}_{L^{2^*}} =\frac{4d}{d-2} \big\|
\varphi\big\|^{2^*}_{L^{2^*}}, \quad  \LLL \big\|  \varphi\big\|^{\frac{2d+2}{d-1}}_{L^{\frac{2d+2}{d-1}}} = \frac{4d}{d-1}
\big\| \varphi\big\|^{\frac{2d+2}{d-1}}_{L^{\frac{2d+2}{d-1}}},
\end{align*}
which implies that
\begin{align*}
\left(\bar{\mu}-\LLL\right)E(\varphi) = & \bar{\mu}E(\varphi)-K(\varphi) \\
 = &
\int_{\R^d} \left(  \frac{2}{d-1} \big|\nabla \varphi\big|^2
+ \frac{2}{d-1} \big|\varphi\big|^{2^*} \right)\; dx,
\\
\LLL  \left(\bar{\mu}-\LLL\right)E(\varphi) = &   \frac{2}{d-1} \LLL \big\|\nabla
\varphi\big\|^2_{L^2} +  \frac{2}{d-1} \LLL
\big\|  \varphi\big\|^{2^*}_{L^{2^*}}  \\
 = &  \int_{\R^d} \left(\frac{8}{d-1}  \big|\nabla
\varphi\big|^2 + \frac{8d}{(d-1)(d-2)} \big| \varphi\big|^{2^*} \right) \; dx.
\end{align*}
This completes the proof.
\end{proof}

According to the above analysis, we will replace the functional
$E$ in \eqref{minimization} with a positive functional $H$, while
extending the minimizing region from ''the mountain ridge
$K(\varphi)=0$'' to ``the mountain flank $K(\varphi)\leq 0$''. Let
\begin{align*}
H(\varphi):= \left(1 - \frac{\LLL}{\bar{\mu}}\right) E(\varphi)
=&\int_{\R^d} \left( \frac1{2d} \big|\nabla \varphi\big|^2 +
\frac1{2d} \big|\varphi\big|^{2^*} \right)\;
dx,
\end{align*}
then for any $\varphi \in H^1 \backslash\{0\}$, we have
\begin{align*}
H(\varphi) > 0 , \quad  \LLL H(\varphi) \geq 0.
\end{align*}

Now by the similar argument to that in \cite{MiaoWZ:NLS:3d Combined}, we can characterization the minimization problem \eqref{minimization} by
making use of $H$.
\begin{lemma}[\cite{MiaoWZ:NLS:3d Combined}, Lemma 2.9]\label{minimization:H} For the minimization $m$ in \eqref{minimization}, we
have
\begin{align}
m  =& \inf  \{ H(\varphi)\; |\;  \varphi \in H^1(\R^d), \; \varphi\not =0,\;
K(\varphi) \leq 0  \} \nonumber\\
 =& \inf \{ H(\varphi)\; |\;\varphi \in H^1(\R^d), \; \varphi\not =0, \;
K(\varphi)< 0 \}. \label{JEqualH}
\end{align}
\end{lemma}%

Next we will use the ($\dot H^1$-invariant) scaling argument to remove the $\dot H^1$-subcritical growth term
$\displaystyle \int_{\R^d} \big|\varphi\big|^{\frac{2d+2}{d-1}}\; dx$
in $K$,
that is, to replace  the constrained
condition $K(\varphi) < 0$ with $K^c(\varphi) < 0$, where
\begin{align*}
K^{c}(\varphi):= \int_{\R^d} \left( 2 |\nabla \varphi|^2
-2|\varphi|^{2^*} \right)\; dx.
\end{align*}

In fact, we have

\begin{lemma}[\cite{MiaoWZ:NLS:3d Combined}, Lemma 2.10]\label{minimization:Hc}For the minimization $m$ in \eqref{minimization}, we
have
\begin{align*} m =& \inf \{ H(\varphi)\; |\; \varphi \in
H^1(\R^d),\; \varphi\not=0, \;
K^c(\varphi) < 0 \}\\
 =& \inf \{ H(\varphi)\; |\; \varphi \in   H^1(\R^d),\; \varphi\not =0, \;
K^c(\varphi) \leq  0 \}.
\end{align*}
\end{lemma}

The above result holds for the defocusing perturbation, which implies that $K^c(\varphi)\leq K(\varphi)$. While the argument does not hold
for the focusing perturbation from the proof in \cite{MiaoWZ:NLS:3d Combined}. Please refer to \cite{AkaIKN:NLS:combined:blowup, AkaIKN:NLS:combined:scattering} for the related discussions.
After these preparations, we can now make use of the sharp constant of the Sobolev inequality in \cite{Aubin:Sharp contant:Sobolev, Talenti:best constant} to compute the minimization $m$ as
following.

\begin{lemma}[\cite{MiaoWZ:NLS:3d Combined}, Lemma 2.11]\label{threshold} For the minimization $m$ in \eqref{minimization}, we
have
\begin{align*}
m=E^c(W).
\end{align*}
\end{lemma}

After the computation of the minimization $m$ in
\eqref{minimization}, we now give some useful variational estimates.

\begin{lemma}[\cite{MiaoWZ:NLS:3d Combined}, Lemma 5.4]\label{L:close of K}
Let $k\in \N$ and $\varphi_0, \ldots, \varphi_k \in H^1(\R^d)$.
Assume that there exist some $\delta$, $\varepsilon>0$ with $ (3d-1)
\varepsilon  < 2d \delta$ such that
\begin{align*}
\sum^k_{j=0}E(\varphi_j) - \varepsilon \leq
E\left(\sum^k_{j=0}\varphi_j\right) < m-\delta,\;\; \text{and}\;\;
 -\varepsilon \leq
K\left(\sum^k_{j=0}\varphi_j\right) \leq
\sum^k_{j=0}K(\varphi_j)+\varepsilon.
\end{align*}
Then $\varphi_j \in \KKK^+$ for all $j=0, \ldots, k$.
\end{lemma}

\begin{lemma} \label{free-energ-equiva} For $d\geq 5$ and any $\varphi \in H^1$ with $K(\varphi)\geq 0$, we
have
\begin{align}\label{free energy}
\int_{\R^d} \left(\frac1{2d} \big|\nabla \varphi \big|^2  + \frac1{2d} \big| \varphi\big|^{2^*}\right)  dx \leq
E(\varphi) \leq \int_{\R^d} \left(\frac12\big|\nabla \varphi \big|^2
  + \frac{d-1}{2d+2} \big| \varphi\big|^{\frac{2d+2}{d-1}}
\right) dx.
\end{align}
\end{lemma}
\begin{proof} On one hand, the second inequality in \eqref{free energy} is trivial.  On the other hand,
by the definition of $E$ and $K$, we have
\begin{align*}
E(\varphi)= \int_{\R^d} \left(\frac1{2d} \big|\nabla \varphi \big|^2 +
\frac1{2d } \big| \varphi\big|^{2^*}\right)\;
dx + \frac{d-1}{4d}K(\varphi),
\end{align*}
which implies the first inequality in \eqref{free energy}.
\end{proof}

At the last of this part, we give the uniform bounds on the
scaling derivative functional $K(\varphi)$ with the   energy
$E(\varphi)$ below the threshold $m$, which plays an important role for
the blow-up and scattering analysis in Section \ref{S:blow up} and Section \ref{S:GWP-Scattering}.

\begin{lemma}[\cite{MiaoWZ:NLS:3d Combined}, Lemma 2.13]\label{uniform bound}
For any $\varphi \in H^1$ with $E(\varphi)<m$, then there exists a
constant $\delta>0$ such that
\begin{enumerate}
\item If $K(\varphi)<0$, then
\begin{align}\label{uniform:K:negative}
K(\varphi) \leq -\bar{\mu}\big(m-E(\varphi)\big).
\end{align}
\item If $K(\varphi)\geq 0$, then
\begin{align}\label{uniform:K:positive}
K(\varphi)\geq \min\left(\bar{\mu}\Big(m-E(\varphi)\Big), \frac{2}{2d-3} \big\|\nabla
\varphi \big\|^2_{L^2} +  \frac{2d}{(d+1)(2d-3)} \big\|\varphi\big\|^{\frac{2d+2}{d-1}}_{L^{
\frac{2d+2}{d-1}}}
\right).
\end{align}
\end{enumerate}
\end{lemma}
\begin{proof} By Lemma \ref{structure:J}, for any $\varphi \in H^1$, we have
\begin{align*}
\LLL^2 E(\varphi) = \bar{\mu} \LLL E(\varphi)-\frac{8}{d-1} \big\|\nabla \varphi\big\|^2_{L^2} - \frac{8d}{(d-1)(d-2)}
\big\|\varphi\big\|^{2^*}_{L^{2^*}}.
\end{align*}
Let $j(\lambda)=E(\varphi^{\lambda}_{d,-2})$, then we have
\begin{align}\label{diff J}
j''(\lambda) =  \bar{\mu}j'(\lambda) - \frac{8 e^{4\lambda}}{d-1} \big\|\nabla \varphi\big\|^2_{L^2} -
\frac{8d e^{\frac{4d}{d-2}\lambda}}{(d-1)(d-2)}\big\|\varphi\big\|^{2^*}_{L^{2^*}}.
\end{align}

\noindent{\bf Case I:} If $K(\varphi)<0$, then by
\eqref{asymptotic:KQ}, Lemma \ref{Postivity:K} and the continuity of $K$ in
$\lambda$, there exists a negative number
$\lambda_0<0$ such that $K(\varphi^{\lambda_0}_{d,-2})=  0$, and
\begin{align*}
K(\varphi^{\lambda}_{d,-2})<   0, \;\; \forall\; \; \lambda\in
(\lambda_0, 0).
\end{align*}
By \eqref{minimization}, we obtain $E(\varphi^{\lambda_0}_{d,-2})
\geq m$. Now by integrating \eqref{diff J} over
$[\lambda_0, 0]$, we have
\begin{align*}
\int^0_{\lambda_0} j''(\lambda)\; d\lambda \leq \bar{\mu}
\int^0_{\lambda_0} j'(\lambda)\; d\lambda,
\end{align*}
which implies that
\begin{align*}
K(\varphi)=j'(0)-j'(\lambda_0)\leq
\bar{\mu}\left(j(0)-j(\lambda_0)\right)\leq -\bar{\mu}
(m-E(\varphi)),
\end{align*}
which implies \eqref{uniform:K:negative}.

\noindent{\bf Case II:} $K(\varphi) \geq 0$. We divide it
into two subcases:

When $\displaystyle 2\bar{\mu} K(\varphi)\geq  \frac{8d}{(d-1)(d-2)}
\big\|\varphi\big\|^{2^*}_{L^{2^*}}.$
Since
\begin{align*}
 \frac{8d}{(d-1)(d-2)}
\big\|\varphi\big\|^{2^*}_{L^{2^*}} = & - \frac{4d}{(d-1)(d-2)}K(\varphi) \\
& + \int_{\R^d}
\left( \frac{8d}{(d-1)(d-2)} \big| \nabla \varphi \big|^2 + \frac{8d^2}{(d+1)(d-1)(d-2)} \big|  \varphi \big|^{\frac{2d+2}{d-1}}
\right) \; dx,
\end{align*}
then we have
\begin{align*}
2\bar{\mu} K(\varphi)\geq & - \frac{4d}{(d-1)(d-2)}K(\varphi) \\
& + \int_{\R^d}
\left( \frac{8d}{(d-1)(d-2)} \big| \nabla \varphi \big|^2 + \frac{8d^2}{(d+1)(d-1)(d-2)} \big|  \varphi \big|^{\frac{2d+2}{d-1}}
\right) \; dx,
\end{align*}
which implies that
\begin{align*}
 K(\varphi)\geq \frac{2}{2d-3} \big\|\nabla
\varphi \big\|^2_{L^2} +  \frac{2d}{(d+1)(2d-3)} \big\|\varphi\big\|^{\frac{2d+2}{d-1}}_{L^{
\frac{2d+2}{d-1}}}.
\end{align*}

When $\displaystyle 2\bar{\mu}  K(\varphi) \leq \frac{8d}{(d-1)(d-2)}
\big\|\varphi\big\|^{2^*}_{L^{2^*}}$.
By \eqref{diff J}, we have for $\lambda=0$
\begin{align}
0< &  2 \bar{\mu}j'(\lambda) <
\frac{8d e^{\frac{4d}{d-2}\lambda}}{(d-1)(d-2)}
\big\|\varphi\big\|^{2^*}_{L^{2^*}}, \nonumber\\
j''(\lambda) = & \bar{\mu}j'(\lambda) - \frac{8 e^{4\lambda}}{d-1} \big\|\nabla \varphi\big\|^2_{L^2} -
\frac{8d e^{\frac{4d}{d-2}\lambda}}{(d-1)(d-2)}\big\|\varphi\big\|^{2^*}_{L^{2^*}} \leq - \bar{\mu}j'(\lambda) . \label{evolution j}
\end{align}
By the continuity of $j'$ and $j''$ in $\lambda$, we know that $j'$ is an accelerated decreasing function as $\lambda$ increases until $j'(\lambda_0)=0$ for some
finite number $\lambda_0>0$ and \eqref{evolution j} holds on $[0,
\lambda_0]$.

By
$
K(\varphi^{\lambda_0}_{d,-2})=j'(\lambda_0)=0,
$
we know that
\begin{align*}
E(\varphi^{\lambda_0}_{d,-2})\geq m.
\end{align*}
Now integrating \eqref{evolution j} over $[0, \lambda_0]$, we obtain that
\begin{align*}
-K(\varphi)=j'(\lambda_0)-j'(0) \leq -\bar{\mu} \big(j(\lambda_0)-j(0)\big)
\leq -\bar{\mu} (m-E(\varphi)).
\end{align*}
This completes the proof.
\end{proof}

%
%
%
%

\section{Part I: Blow up for $\KKK^-$}\label{S:blow up} In this section, we prove the
blow-up result of Theorem \ref{theorem} in the case that $u_0$ is radial. The case $xu_0\in L^2$ is trivial.
 We can also refer to \cite{MiaoWZ:NLS:3d Combined} for the similar discussions to the case
 $d=3$. The spatial localization argument comes from
 \cite{OgaTsu:Blowup:NLS:91}. Now let $\phi$ be a smooth, radial function satisfying $
\partial^2_r \phi(r) \leq 2$, $\phi(r)=r^2$ for $r\leq 1$, and
$\phi(r)$ is constant for $r\geq 3$. For some $R$, we define
\begin{align*}
V_R(t):=\int_{\R^d} \phi_R(x) |u(t,x)|^2\; dx, \quad
\phi_R(x)=R^2\phi\left(\frac{|x| }{R }\right).
\end{align*}

By Lemma \ref{L:virial},
$
\Delta\phi_R(r)= 2d $  for $ r\leq R,$ and $\Delta^2 \phi_R(r)=0 $ for $  r\leq R,
$
we have
\begin{align*}
\partial^2_t V_R(t)
= &  \;  4 \int_{\R^d}\partial_{ij}(\phi_R)  u_i(t,x)\bar{u}_j(t,x)\;  dx
- \int_{\R^d} \Delta^2 \phi_R \big| u(t,x)\big|^2\; dx\\
& - \frac{4}{d}\int_{\R^d} \Delta \phi_R \big| u(t,x)\big|^{2^*}\; dx+\frac{4}{d+1}\int_{\R^3} \Delta \phi_R \big| u(t,x)\big|^{\frac{2d+2}{d-1}}\; dx\\
\leq &\; 4 \int_{\R^d} \left( 2 |\nabla u (t)|^2 -2|u (t)|^{2^*} +\frac{2d}{d+1}
|u (t)|^{\frac{2d+2}{d-1}} \right)\; dx \\
 & +  \frac{c}{R^2}\int_{R\leq |x|\leq 3R} \big| u (t) \big|^2 \; dx + c \int_{R\leq |x|\leq 3R} \left( \big| u (t) \big|^{\frac{2d+2}{d-1}}
 + \big| u(t) \big|^{2^*} \right) \; dx.
\end{align*}

By the radial Sobolev inequality, we have
\begin{align*}
\big\|f\big\|_{L^\infty(|x|\geq R)} \leq & \frac{c}{R^{(d-1)/2}} \big\|f\big\|^{1/2}_{L^2(|x|\geq R)}\big\|\nabla f\big\|^{1/2}_{L^2(|x|\geq R)}.
\end{align*}
Therefore, by the mass conservation and Young's inequality, we know that for any $\epsilon>0$ there exist sufficiently large $R$ such that
\begin{align}
 &  \partial^2_t   V_R(t) \nonumber\\
 \leq & 4 K(u(t))
+ \epsilon \big\|\nabla u(t)\big\|^2_{L^2 } + \epsilon^2. \nonumber\\
= & \frac{16d}{d-2} E(u) - \left(\frac{16}{d-2}-\epsilon\right)\big\|\nabla u(t)\big\|^2_{L^2} - \frac{8d}{(d+1)(d-2)}\big\|u(t)\big\|^{\frac{2d+2}{d-1}}_{L^{\frac{2d+2}{d-1}}} + \epsilon^2. \label{videntity:second der}
\end{align}
By $K(u)<0$, the mass and energy conservations, Lemma \ref{uniform bound} and
the continuity argument, we know that for any $t\in I$, we have
\begin{align*}
K(u(t)) \leq -\bar{\mu}\left(m-E(u(t))\right)<0.
\end{align*}
By Lemma \ref{minimization:H}, we have
\begin{align*}
m \leq H(u(t))< \frac1d \big\|u(t)\big\|^{2^*}_{L^{2^*}}.
\end{align*}
where we have used the fact that $K(u(t))<0$ in the second inequality. By the fact $m=\frac1d \left(C^*_d\right)^{-d}$ and the sharp Sobolev inequality, we have
\begin{align*}
\big\|\nabla u(t)\big\|^{2^*}_{L^2} \geq \left(C^*_d\right)^{-2^*} \big\|u(t)\big\|^{2^*}_{L^{2^*}} > \left(dm\right)^{\frac{d}{d-2}},
\end{align*}
which implies that $\big\|\nabla u(t)\big\|^2_{L^2} > dm$.

In addition, by $E(u_0)<m$ and energy conservation, there exists $\delta_1>0$ such that
$E(u(t))\leq (1-\delta_1)m$. Thus, if we choose $\epsilon$ sufficiently small, we have
\begin{align*}
\partial^2_t V_R(t) \leq \frac{16d}{d-2} (1-\delta_1)m - d\big(\frac{16}{d-2}-\epsilon\big) m +   \epsilon^2 \leq -  \frac{8d}{d-2} \delta_1 m,
\end{align*}
which implies that $u$ must blow up at finite time. \qed

%
%
%
%

\section{Profile decomposition}\label{S:profiledec} In this part, we will use the method in \cite{BahG:NLW:proffile decomp, IbrMN:f:NLKG, Ker:NLS:profile decomp, MiaoWZ:NLS:3d Combined}
to show the linear and nonlinear profile
decompositions of the $H^1$-bounded sequences of solutions of \eqref{NLS} in five and higher dimensions, which will be used to construct the critical
element (minimal  energy non-scattering solution) and show its properties, especially the compactness and regularity.
In order to do it, we cite the similar notation to those in \cite{IbrMN:f:NLKG, MiaoWZ:NLS:3d Combined}. Now we introduce the complex-valued function
$\overrightarrow{v}(t,x)$ by
\begin{align*}
\overrightarrow{v}(t,x)=\left<\nabla\right>v (t,x), \quad v (t,x)
=\left<\nabla\right>^{-1}\overrightarrow{v} (t,x).
\end{align*}

Given $(t^j_n, x^j_n, h^j_n)\in \R \times \R^d \times (0, 1]$, let
$\tau^j_n$, $T^j_n$ denote the scaled time drift, the unitary
operator in $L^2(\R^d)$, defined by
\begin{align*}
\tau^j_n = - \frac{t^j_n}{\left(h^j_n\right)^2}, \quad T^j_n \varphi(x) =
\frac{1}{(h^j_n)^{d/2}} \varphi \left(\frac{x-x^j_n}{h^j_n}\right).
\end{align*}

\subsection{Linear profile decomposition} By the similar arguments to that in \cite{MiaoWZ:NLS:3d Combined}, we can establish that

\begin{proposition}[\cite{MiaoWZ:NLS:3d Combined}, Proposition 5.1 and Lemma 5.3]\label{L:linear profile}
Let
$$\overrightarrow{v}_n(t,x)=e^{it\Delta}\overrightarrow{v}_n(0)$$
 be a sequence of the free
Schr\"{o}dinger solutions with bounded $L^2$ norm. Then up to a
subsequence, there exist $K\in \{0, 1,2,\ldots, \infty\}$,
$\{\varphi^j\}_{j\in [0, K)}\subset L^2(\R^d)$ and $\{t^j_n, x^j_n,
h^j_n\}_{n\in \N} \subset \R \times \R^d \times (0, 1]$ satisfying
\begin{align}\label{profile:linear}
\overrightarrow{v}_n(t,x) = \sum^{k-1}_{j=0}
\overrightarrow{v}^j_n(t,x) + \overrightarrow{w}^k_n(t,x),
\end{align}
where $ \overrightarrow{v}^j_n (t,x) = e^{i(t-t^j_n)\Delta} T^j_n
\varphi^j$, and
\begin{align}\label{small:w weak topology}
\lim_{k\rightarrow K} \varlimsup_{n\rightarrow +\infty}
\big\|\overrightarrow{w}^k_n\big\|_{L^{\infty}_t(\R;
B^{-d/2}_{\infty, \infty}(\R^d))} =0,
\end{align}
and for any
$l<j<k\leq K$,
\begin{align}
\lim_{n\rightarrow +\infty} \left(\log \left| \dfrac{h^j_n}{h^l_n} \right|
+ \left| \frac{t^j_n - t^l_n}{(h^l_n)^2} \right| +   \left|
\frac{x^j_n
- x^l_n}{ h^l_n } \right|\right)=\infty,\label{orth:I}\\
\lim_{k\rightarrow K}  \varlimsup_{n\rightarrow +\infty} \left|
M(v_n(0)) - \sum^{k-1}_{j=0} M(v^j_n(0)) -M(w^k_n(0))\right| =0, \label{orth:mass}\\
\lim_{k\rightarrow K}  \varlimsup_{n\rightarrow +\infty} \left|
E(v_n(0)) - \sum^{k-1}_{j=0}E(v^j_n(0)) - E (w^k_n(0))\right| =0, \label{orth:energy} \\
\lim_{k\rightarrow K} \varlimsup_{n\rightarrow +\infty} \left|
K(v_n(0)) - \sum^{k-1}_{j=0} K(v^j_n(0)) -K(w^k_n(0))\right| =0. \label{orth:scalingderivative}
\end{align}
Moreover, each sequence $\{h^j_n\}_{n\in \N}$ is either going to $0$
or identically $1$ for all $n$.
\end{proposition}

We call $\overrightarrow{v}^j_n$ and $\overrightarrow{w}^k_n$ the free concentrating wave
and the remainder, respectively. According to the above result and Lemma \ref{L:close of K}, we conclude
\begin{proposition}[\cite{MiaoWZ:NLS:3d Combined}, Proposition 5.5]\label{decomp:stable:K}
Let $\overrightarrow{v}_n(t,x) $ be a sequence of the free Schr\"{o}dinger solutions satisfying
$
v_n(0)\in \KKK^+ \;\; \text{and}\;\;  E (v_n(0))<m.
$
Let
\begin{align*}
\overrightarrow{v}_n(t,x) = \sum^{k-1}_{j=0}
\overrightarrow{v}^j_n(t,x) + \overrightarrow{w}^k_n(t,x),
\end{align*}
be the linear profile decomposition given by Proposition
\ref{L:linear profile}. Then for large $n$ and all $j<K$, we have
\begin{align*}
v^j_n(0)\in \KKK^+,\;\;  \;\; w^K_n(0) \in \KKK^+,
\end{align*}
such that \eqref{orth:mass}-\eqref{orth:scalingderivative}. Moreover for all $j<K$, we have
\begin{align*}
0 \leq  \varliminf_{ n\rightarrow +\infty} E(v^j_n(0)) \leq
\varlimsup_{n\rightarrow +\infty} E(v^j_n(0)) \leq
\varlimsup_{n\rightarrow +\infty} E(v_n(0)),
\end{align*}
where the last inequality becomes equality only if $K=1$ and $w^1_n
\rightarrow 0$ in $L^{\infty}_t \dot H^1_x$.
\end{proposition}

\subsection{Nonlinear profile decomposition} After the linear profile decomposition of a sequence of initial data in
the last subsection, we now give the nonlinear profile decomposition of a sequence
of solutions of \eqref{NLS} with the same initial data in the energy
space $H^1(\R^d)$. The procedure is the same as the 3d case in \cite{MiaoWZ:NLS:3d Combined}.

Let $v_n(t,x)$ be a sequence of solutions for the free
Schr\"{o}dinger equation with initial data in $\KKK^+$, that is,
$v_n \in H^1(\R^d)$ and
\begin{align*}
\left(i\partial_t + \Delta\right) v_n =0,\quad v_n(0)\in \KKK^+.
\end{align*}

Let
\begin{align*}
\overrightarrow{v}_n(t,x) = \left<\nabla\right> v_n(t,x),
\end{align*}
then by Proposition \ref{L:linear profile}, we have a sequence of the
free concentrating wave $\overrightarrow{v}^j_n(t,x)$ with
$\overrightarrow{v}^j_n (t^j_n) = T^j_n \varphi^j$, $v^j_n(0)\in
\KKK^+$ for $j=0, \ldots, K$, such that
\begin{align*}
\overrightarrow{v}_n(t,x) = & \sum^{k-1}_{j=0}
\overrightarrow{v}^j_n(t,x) + \overrightarrow{w}^k_n(t,x)  =
\sum^{k-1}_{j=0} e^{i(t-t^j_n)\Delta}T^j_n \varphi^j
  + \overrightarrow{w}^k_n \\
  = &  \sum^{k-1}_{j=0} T^j_n e^{i\left(\frac{t-t^j_n}{(h^j_n)^2}\right)\Delta} \varphi^j
  + \overrightarrow{w}^k_n .
\end{align*}

Now for any concentrating wave $\overrightarrow{v}^j_n $, $j=0,
\ldots, K$, we undo the group action, i.e., the scaling and translation transformation $T^j_n$,  to look for the
linear profile $V^j$. Let
\begin{align*}
\overrightarrow{v}^j_n(t,x) = &  T^j_n
\overrightarrow{V}^j\left(\frac{t-t^j_n}{(h^j_n)^2}\right),
\end{align*}
then we have
\begin{align*}
 \left(i\partial_t + \Delta\right) \overrightarrow{V}^j =0, \quad
  \overrightarrow{V}^j(0)=\varphi^j.
\end{align*}

Now let $u^j_n(t,x)$ be the nonlinear solution of \eqref{NLS} with
initial data $v^j_n(0)$, that is
\begin{align*}
 \left(i\partial_t + \Delta\right) \overrightarrow{u}^j_n(t,x) = & \left< \nabla\right> f_1 (\left<\nabla\right> ^{-1} \overrightarrow{u}^j_n )+ \left< \nabla\right> f_2 (\left<\nabla\right> ^{-1} \overrightarrow{u}^j_n ),\\
  \quad \overrightarrow{u}^j_n (0)=& \overrightarrow{v}^j_n (0)= T^j_n\overrightarrow{V}^j(\tau^j_n),\quad  u^j_n (0)\in
  \KKK^+,
\end{align*}
where $\tau^j_n = - t^j_n/ (h^j_n)^2$. In order to look for the nonlinear
profile $\overrightarrow{U}^j_{\infty}$ associated to the free concentrating wave
$\left(\overrightarrow{v}^j_n;\; h^j_n, t^j_n, x^j_n \right)$, we
also need undo the group action. We denote
\begin{align*}
   \overrightarrow{u}^j_n(t,x) =& T^j_n
\overrightarrow{U}^j_n\left(\frac{t-t^j_n}{(h^j_n)^2}\right),
\end{align*}
then we have
\begin{align*}
 \left(i\partial_t + \Delta\right) \overrightarrow{U}^j_n = &
\left( \left<\nabla\right>^{j}_{n}\right) f_1 \left( \left(
\left<\nabla\right>^{j}_{n}\right) ^{-1} \overrightarrow{U}^j_n
\right) + \left(h^j_n\right)^{\frac2{d-1}} \cdot \left( \left<\nabla\right>^{j}_{n}\right) f_2
\left( \left( \left<\nabla\right>^{j}_{n}\right) ^{-1}
\overrightarrow{U}^j_n \right),\\
\overrightarrow{U}^j_n(\tau^j_n) =& \overrightarrow{V}^j(\tau^j_n).
\end{align*}

Up to a subsequence, we may assume that there exist
$h^j_{\infty}\in \{0, 1\}$ and $\tau^j_{\infty} \in [-\infty, \infty]$ for
every $j=\{0, \ldots, K\}$, such that
\begin{align*}
h^j_n \rightarrow   \; h^j_{\infty} ,\;\; \text{and}\;\;  \tau^j_n
\rightarrow \; \tau^j_{\infty}  .
\end{align*}
As $n\rightarrow +\infty$, the limit equation of $\overrightarrow{U}^j_n$
is given by
\begin{align*}
 \left(i\partial_t + \Delta\right) \overrightarrow{U}^j_{\infty} = &
\left( \left<\nabla\right>^{j}_{\infty}\right) f_1 \left( \left(
\left<\nabla\right>^{j}_{\infty}\right) ^{-1}
\overrightarrow{U}^j_{\infty} \right) + \left(h^j_{\infty}\right)^{\frac{2}{d-1}} \cdot \left(
\left<\nabla\right>^{j}_{\infty}\right) f_2 \left( \left(
\left<\nabla\right>^{j}_{\infty}\right) ^{-1}
\overrightarrow{U}^j_{\infty} \right),\\
\overrightarrow{U}^j_\infty(\tau^j_\infty) =&
\overrightarrow{V}^j(\tau^j_\infty) \in L^2(\R^d).
\end{align*}

Let
\begin{align*}
\widehat{U}^j_{\infty}:=
\left(\left<\nabla\right>^j_{\infty}\right)^{-1}\overrightarrow{U}^j_{\infty},
\end{align*}
then
\begin{align*}
 \left(i\partial_t + \Delta\right) \widehat{U}^j_{\infty} = &
  f_1 \left(\widehat{U}^j_{\infty} \right) + \left( h^j_{\infty}\right)^{\frac{2}{d-1}} \cdot
  f_2 \left(\widehat{U}^j_{\infty} \right),\\
 \widehat{U}^j_{\infty}(\tau^j_\infty) =& \left(\left<\nabla\right>^j_{\infty}\right)^{-1}
\overrightarrow{V}^j(\tau^j_\infty).
\end{align*}

The unique existence of a local solution
$\overrightarrow{U}^j_{\infty}$ around $\tau^j_{\infty}$ is known in
all cases, including $h^j_{\infty}=0$ and $\tau^j_{\infty}=\pm
\infty$. $\overrightarrow{U}^j_{\infty} $ on the maximal existence
interval is called the nonlinear profile associated with the free
concentrating wave $\left(\overrightarrow{v}^j_n;\; h^j_n, t^j_n,
x^j_n \right)$.

The nonlinear concentrating wave $u^j_{(n)}$ associated with
$\left(\overrightarrow{v}^j_n;\; h^j_n, t^j_n, x^j_n \right)$ is
defined by
\begin{align*}
\overrightarrow{u}^j_{(n)}(t,x)=T^j_n
\overrightarrow{U}^j_{\infty}\left(\frac{t-t^j_n}{(h^j_n)^2}\right),
\end{align*}
then we have
\begin{align*}
 \left(i\partial_t + \Delta\right) \overrightarrow{u}^j_{(n)}
=& \left<\nabla\right>^j_{\infty} f_1 \left( \left(
\left<\nabla\right>^j_{\infty}\right)^{-1}
\overrightarrow{u}^j_{(n)} \right) +  \left(h^j_{\infty}\right)^{\frac{2}{d-1}} \cdot
\left<\nabla\right>^j_{\infty} f_2 \left( \left(
\left<\nabla\right>^j_{\infty} \right)^{-1}
\overrightarrow{u}^j_{(n)} \right), \\
 \overrightarrow{u}^j_{(n)}(0) =&
T^j_n\overrightarrow{U}^j_{\infty}(\tau^j_n),
\end{align*}
which implies that
\begin{align*}
\big\|\overrightarrow{u}^j_{(n)}(0) - \overrightarrow{u}^j_{n}(0)
\big\|_{L^2} = & \big\|T^j_n\overrightarrow{U}^j_{\infty}(\tau^j_n)
- T^j_n\overrightarrow{V}^j(\tau^j_n) \big\|_{L^2}
=\big\|
\overrightarrow{U}^j_{\infty}(\tau^j_n) -
 \overrightarrow{V}^j(\tau^j_n) \big\|_{L^2}
\\
\leq & \big\| \overrightarrow{U}^j_{\infty}(\tau^j_n) -
\overrightarrow{U}^j_{\infty}(\tau^j_{\infty}) \big\|_{L^2} + \big\|
 \overrightarrow{V}^j(\tau^j_n)- \overrightarrow{V}^j(\tau^j_{\infty})
 \big\|_{L^2}   \rightarrow  0.
\end{align*}

We denote
\begin{align*}
\overrightarrow{u}^j_{(n)} = \left<\nabla \right>u^j_{(n)}.
\end{align*}
If $h^j_{\infty}=1$, we have $h^j_n\equiv1$, then $ u^j_{(n)} \in
H^1(\R^d)$ and satisfies
\begin{align*}
 \left(i\partial_t + \Delta\right) u^j_{(n)} = f_1(u^j_{(n)}) + f_2
 (u^j_{(n)}).
\end{align*}
If $h^j_{\infty}=0$, then $ u^j_{(n)} \in H^1(\R^d)$ satisfies
\begin{align*}
 \left(i\partial_t + \Delta\right) u^j_{(n)} =
 \frac{|\nabla|}{\left<\nabla\right>}f_1\left(\frac{\left<\nabla\right>}{|\nabla|}u^j_{(n)}\right).
\end{align*}

Let $u_n$ be a sequence of (local) solutions of \eqref{NLS} with
initial data in $\KKK^+$ at $t=0$, and let $v_n$ be the sequence of
the free solutions with the same initial data. We consider the
linear profile decomposition given by Proposition \ref{L:linear
profile}
\begin{align*}
\overrightarrow{v}_n(t,x) = \sum^{k-1}_{j=0}
\overrightarrow{v}^j_n(t,x) +
\overrightarrow{w}^k_n(t,x),\quad\overrightarrow{v}^j_n (t^j_n) =
T^j_n \varphi^j, \quad  v^j_n(0)\in \KKK^+.
\end{align*}
With each free concentrating wave $\{\overrightarrow{v}^j_n\}_{n\in
\N}$, we associate the nonlinear concentrating wave
$\{\overrightarrow{u}^{j}_{(n)}\}_{n\in \N}$. A nonlinear profile
decomposition of $u_n$ is given by
 \begin{align}\label{profile:nonlinear}
\overrightarrow{u}^{<k}_{(n)}(t,x):=\sum^{k-1}_{j=0}\overrightarrow{u}^j_{(n)}(t,x)
= \sum^{k-1}_{j=0} T^j_n
\overrightarrow{U}^j_{\infty}\left(\frac{t-t^j_n}{(h^j_n)^2}\right).
\end{align}

Since the smallness condition \eqref{small:w weak topology} and the orthogonality condition \eqref{orth:I} ensure that
every nonlinear concentrating wave and the remainder interacts weakly with the others,
we will show that $\overrightarrow{u}^{<k}_{(n)}+
\overrightarrow{w}^k_n$ is a good approximation for
$\overrightarrow{u}_n$ provided that each nonlinear profile has the finite
global Strichartz norm.

Now we define the Strichartz norms. First let $ST(I)$ and $ST^*(I)$ be the function spaces on $I \times \R^d$
defined as Section \ref{S:LWP and Pertubation}
\begin{align*}
ST(I):=  W_1(I) \cap W_2(I), \quad
ST^*(I):=   L^2_t \left(I;
B^{\frac2d}_{\frac{2d^2}{d^2+4}, 2} \right).
\end{align*}
The Strichartz norm for the nonlinear profile $\widehat{U}^j_{\infty}$ depends on the scaling
$h^j_{\infty}$.
\begin{align*}
ST^j_{\infty}(I):=\begin{cases}  W_1(I) \cap W_2(I), \quad & \text{for}\; h^j_{\infty}=1, \\
W_1(I), \quad & \text{for}\;
h^j_{\infty}=0.
\end{cases}
\end{align*}

By the similar arguments to that in \cite{IbrMN:f:NLKG, MiaoWZ:NLS:3d Combined}, we have

\begin{lemma}[\cite{MiaoWZ:NLS:3d Combined}, Lemma 5.6] \label{orth:III}
In the nonlinear profile decomposition
\eqref{profile:nonlinear}. Suppose that for each $j<K$, we have
\begin{align*}
\big\|\widehat{U}^j_{\infty}\big\|_{ST^j_{\infty}(\R)}
+\big\|\overrightarrow{U}^{j}_{\infty}\big\|_{L^{\infty}_tL^2_x(\R^d)}<\infty.
\end{align*}
Then for any finite interval $I$, any $j<K$
and any $k\leq K$, we have
\begin{align}
\varlimsup_{n\rightarrow +\infty}  \big\|u^j_{(n)}\big\|_{ST(I)}
\lesssim & \big\|\widehat{U}^j_{\infty}\big\|_{ST^j_{\infty}(\R)},\label{approximation:ST control:single} \\
\varlimsup_{n\rightarrow +\infty}
\big\|u^{<k}_{(n)}\big\|^2_{ST(I)}\lesssim
&\varlimsup_{n\rightarrow +\infty}  \sum_{j<k}
\big\|u^j_{(n)}\big\|^2_{ST(I)}, \label{approximation:ST control:all}
\end{align}
where the implicit constants do not depend on $I, j$ or $k$. We
also have
\begin{align}
\lim_{n\rightarrow +\infty} \left\| f_1\left(u^{<k}_{(n)}\right)   - \sum_{j<k}
\frac{\left<\nabla\right>^j_{\infty}}{\left<\nabla\right>}  f_1 \left(
\frac{\left<\nabla\right>}{\left<\nabla\right>^j_{\infty}}  u^j_{(n)} \right)  \right\|_{ST^*(I)}=0, \label{orth:ST norm:I}\\
\lim_{n\rightarrow +\infty} \left\| f_2\left(u^{<k}_{(n)}\right)   - \sum_{j<k}  \left( h^j_{\infty}\right)^{\frac{2}{d-1}}
\frac{\left<\nabla\right>^j_{\infty}}{\left<\nabla\right>}  f_2 \left(
\frac{\left<\nabla\right>}{\left<\nabla\right>^j_{\infty}}  u^j_{(n)} \right)  \right\|_{ST^*(I)}=0.  \label{orth:ST norm:II}
\end{align}
\end{lemma}

After this preliminaries, we now show that $\overrightarrow{u}^{<k}_{(n)}+
\overrightarrow{w}^k_n$ is a good approximation for
$\overrightarrow{u}_n$ provided that each nonlinear profile has finite
global Strichartz norm.

\begin{proposition}[\cite{MiaoWZ:NLS:3d Combined}, Proposition 5.7]\label{approximation}
Let $u_n$ be a sequence of local solutions of \eqref{NLS} around
$t=0$ in $\KKK^+$ satisfying \begin{align*}
M\left(u_n\right)< \infty, \quad \varlimsup_{n\rightarrow
\infty} E(u_n)<m.
\end{align*} Suppose that in the nonlinear profile decomposition
\eqref{profile:nonlinear}, every nonlinear profile
$\widehat{U}^j_{\infty}$ has finite global Strichartz and energy
norms we have
\begin{align*}
\big\|\widehat{U}^j_{\infty}\big\|_{ST^j_{\infty}(\R)}
+\big\|\overrightarrow{U}^{j}_{\infty}\big\|_{L^{\infty}_tL^2_x(\R^d)}<\infty.
\end{align*}
 Then $u_n$ is bounded for large $n$ in the Strichartz and the
 energy norms
\begin{align*}\varlimsup_{n\rightarrow \infty}
\big\|u_n\big\|_{ST(\R)} +
\big\|\overrightarrow{u}_n\big\|_{L^{\infty}_tL^2_x(\R)} < \infty.
\end{align*}
\end{proposition}

\begin{proof}We only need to verify the condition of Proposition \ref{stability}. Note that $u^{<k}_{(n)} + w^k_n$ satisfies that
\begin{align*}
  & \left(i \partial_t + \Delta\right) \left(u^{<k}_{(n)} + w^k_n\right)=f_1\left(u^{<k}_{(n)} + w^k_n\right) + f_2\left( u^{<k}_{(n)} + w^k_n\right)\\
 &\qquad + f_1\left(u^{<k}_{(n)}\right) - f_1\left(u^{<k}_{(n)}+ w^k_n \right)
  + f_2\left(u^{<k}_{(n)}\right) - f_2\left(u^{<k}_{(n)}+ w^k_n \right)\\
& \qquad+ \sum_{j<k}
\frac{\left<\nabla\right>^j_{\infty}}{\left<\nabla\right>}  f_1 \left(
\frac{\left<\nabla\right>}{\left<\nabla\right>^j_{\infty}}  u^j_{(n)} \right) - f_1\left(u^{<k}_{(n)}\right)   + \sum_{j<k}  \left(h^j_{\infty}\right)^{\frac{2}{d-1}}
\frac{\left<\nabla\right>^j_{\infty}}{\left<\nabla\right>}  f_2 \left(
\frac{\left<\nabla\right>}{\left<\nabla\right>^j_{\infty}}  u^j_{(n)} \right) -  f_2\left(u^{<k}_{(n)}\right).
\end{align*}

First, by the construction of $\overrightarrow{u}^{<k}_{(n)}$, we know that
\begin{align*}
\left\| \left(\overrightarrow{u}^{<k}_{(n)} (0) + \overrightarrow{w}^k_n(0)\right) -\overrightarrow{u}_n(0) \right\|_{ L^2_x}
\leq \sum_{j<k} \left\|\overrightarrow{u}^j_{(n)}(0) -\overrightarrow{u}^j_n(0) \right\|_{ L^2_x} \rightarrow 0,
\end{align*}
as $n\rightarrow +\infty$, which also implies that for large $n$, we have
\begin{align*}
\left\| \overrightarrow{u}^{<k}_{(n)} (0) + \overrightarrow{w}^k_n(0)  \right\|_{ L^2_x} \leq E_0.
\end{align*}

Next, by the linear profile decomposition in Proposition \ref{L:linear profile}, we know that
\begin{align*}
\big\|u_n(0)\big\|^2_{L^2}= &  \left\|v_n(0)\right\|^2_{L^2_x}
= \sum_{j<k}\left\|v^j_n(0)\right\|^2_{L^2_x} + \left\|w^k_n(0)\right\|^2_{L^2_x} + o_n(1)\\
\geq &  \sum_{j<k}\left\|v^j_n(0)\right\|^2_{L^2_x} + o_n(1)
=     \sum_{j<k}\left\|u^j_{(n)}(0)\right\|^2_{L^2_x} + o_n(1),
\\
\left\|u_n(0)\right\|^2_{\dot H^1_x}=& \left\|v_n(0)\right\|^2_{\dot H^1_x}
= \sum_{j<k}\left\|v^j_n(0)\right\|^2_{\dot H^1_x} + \left\|w^k_n(0)\right\|^2_{\dot H^1_x} + o_n(1)\\
\geq &  \sum_{j<k}\left\|v^j_n(0)\right\|^2_{\dot H^1_x} + o_n(1)
=     \sum_{j<k}\left\|u^j_{(n)}(0)\right\|^2_{\dot H^1_x} + o_n(1),
\end{align*}
which means except for a finite set $J\subset \N$, the energy of $u^j_{(n)}$ with $j\not \in J$ is smaller than the iteration
threshold, hence we have
\begin{align*}
\big\|u^j_{(n)}\big\|_{ST(\R)} \lesssim \big\|\overrightarrow{u}^j_{(n)}(0)\big\|_{L^2_x},
\end{align*}
thus, for any finite interval $I$, by Lemma \ref{orth:III}, we have
\begin{align*}
\sup_k \varlimsup_{n\rightarrow +\infty} \big\|u^{<k}_{(n)}\big\|^2_{ST(I)} \lesssim &  \sup_k \varlimsup_{n\rightarrow +\infty}\sum_{j<k} \big\|u^{j}_{(n)}\big\|^2_{ST(I)}\\
= & \sup_k \varlimsup_{n\rightarrow +\infty} \left[\sum_{j<k, j\in J} \big\|u^{j}_{(n)}\big\|^2_{ST(I)} + \sum_{j<k, j\not \in J} \big\|u^{j}_{(n)}\big\|^2_{ST(I)}\right]\\
\lesssim & \sum_{j<k, j\in J} \big\|\widehat{U}^{j}_{\infty}\big\|^2_{ST^j_{\infty}(I)} + \sup_k \varlimsup_{n\rightarrow +\infty} \sum_{j<k, j\not \in J} \big\|\overrightarrow{u}^j_{(n)}(0)\big\|^2_{L^2_x}\\
< & \infty.
\end{align*}
This together with the Strichartz estimate for $w^k_n$ implies that
\begin{align*}
\sup_k \varlimsup_{n\rightarrow +\infty} \big\|u^{<k}_{(n)} + w^k_n\big\|^2_{ST(I)} <  \infty.
\end{align*}

Last we need show the nonlinear perturbation is small in some sense. By Proposition \ref{L:linear profile} and Lemma \ref{orth:III}, we have
\footnote{Since we use the full derivative $\left<\nabla\right>$, we need use the local smoothing effect about the free solution $w(t,x)$
as that in \cite{AkaIKN:NLS:combined:scattering, KiV:en-NLS:high dim} to verify the weak interaction between the nonlinear concentrating waves $u^j_{(n)}(t,x)$ and the remainder $w(t,x)$.}
\begin{align*}
& \left\| f_1\left(u^{<k}_{(n)}\right) - f_1\left(u^{<k}_{(n)}+ w^k_n \right)  \right\|_{ST^*(I)} \rightarrow 0 ,\\
 &  \left\| f_2\left(u^{<k}_{(n)}\right) - f_2\left(u^{<k}_{(n)}+ w^k_n \right) \right\|_{ST^*(I)} \rightarrow 0,
 \end{align*}
and
\begin{align*}
&  \left\| \sum_{j<k}
\frac{\left<\nabla\right>^j_{\infty}}{\left<\nabla\right>}  f_1 \left(
\frac{\left<\nabla\right>}{\left<\nabla\right>^j_{\infty}}  u^j_{(n)} \right) - f_1\left(u^{<k}_{(n)}\right)  \right\|_{ST^*(I)} \rightarrow 0,  \\
&   \left\| \sum_{j<k}  h^j_{\infty}
\frac{\left<\nabla\right>^j_{\infty}}{\left<\nabla\right>}  f_2 \left(
\frac{\left<\nabla\right>}{\left<\nabla\right>^j_{\infty}}  u^j_{(n)} \right) -  f_2\left(u^{<k}_{(n)}\right)  \right\|_{ST^*(I)} \rightarrow 0,
\end{align*}
as $n\rightarrow +\infty$. Therefore, by Proposition \ref{stability}, we can obtain the desired result, which concludes the proof.
\end{proof}

%
%
%
%

\section{Part II: GWP and Scattering for $\KKK^+$}\label{S:GWP-Scattering}
We now use the stability analysis of the
scattering solution of \eqref{NLS} and the compactness analysis  of
a sequence of the energy solutions of \eqref{NLS} to show the scattering result of Theorem
\ref{theorem} by contradiction.

For any finite positive number $C<\infty$, let $E^*_C$ be the threshold for the uniform Strichartz norm bound,
i.e.,
\begin{align*}
E^*_C:=\sup\{A>0, ST(A)<\infty\}
\end{align*}
where $ST(A)$ denotes the supremum of $\big\|u\big\|_{ST(I)}$ for any
strong solution $u$ of \eqref{NLS} in $\KKK^+$ on any interval $I$ satisfying
$E(u)\leq A$, $M(u)\leq C$.

The small solution scattering theory gives us $E^*_C>0$. We are going to show that $E^*_C\geq m$ by contradiction.
From now on, suppose that $E^*_C\geq m$ fails,
that is, we assume that
\begin{align}\label{contradiction:J}
E^*_C<m.
\end{align}

\subsection{Existence of a critical element} This part is similar to section 6.1 in \cite{MiaoWZ:NLS:3d Combined}.
By the definition of $E^*_C$ and the fact that $E^*_C<m$, there exist
a sequence of solutions $\{u_n\}_{n\in \N}$  of \eqref{NLS} in $\KKK^+$, which have the maximal existence interval $I_n$ and satisfy that for some finite number $C$
\begin{align*}
M(u_n) \leq C,
\quad E(u_n)\rightarrow E^*_C<m, \quad \big\|u_n\big\|_{ST(I_n)}\rightarrow
+\infty, \quad \text{as}\;\; n\rightarrow +\infty,
\end{align*}
then we have $\big\|u_n\big\|_{H^1}<\infty$ by Lemma \ref{free-energ-equiva}. By the compact argument (profile decomposition) and the stability theory, we can show that

\begin{theorem}\label{APS:existence} For $d\geq 5$. Let $u_n$ be a sequence of solutions of \eqref{NLS} in $\KKK^+$ on
$I_n\subset \R$ satisfying
\begin{align*}
M(u_n) \leq  C,
\quad E(u_n)\rightarrow E^*_C<m, \quad \big\|u_n\big\|_{ST(I_n)}\rightarrow
+\infty, \quad \text{as}\;\; n\rightarrow +\infty.
\end{align*}
Then there exists a global solution $u_c$ of \eqref{NLS} in $\KKK^+$
satisfying
\begin{align*}
E(u_c)=E^*_C<m, \quad K(u_c) > 0, \quad \big\|u_c\big\|_{ST(\R)}=\infty.
\end{align*}
In addition, there are a sequence $(t_n, x_n)\in \R \times \R^d$ and
$\varphi \in L^2(\R^d)$ such that, up to a subsequence,  we have as $n\rightarrow +\infty$,
\begin{align}\label{compact:initial}
\left\|\frac{|\nabla|}{\left<\nabla \right>}\Big(\overrightarrow{u}_n(0, x) -
e^{-it_n\Delta}\varphi(x-x_n)\Big)\right\|_{L^2}\rightarrow 0.
\end{align}
\end{theorem}

\begin{proof} By the time translation symmetry of \eqref{NLS}, we can translate $u_n$ in $t$ such that $0\in I_n$ for all $n$.
Then by the linear and nonlinear profile decomposition of $u_n$, we have
\begin{align*}
e^{it\Delta }\overrightarrow{u}_n(0,x) = &
\sum_{j<k}\overrightarrow{v}^j_n(t,x) + \overrightarrow{w}^k_n(t,x),
 \quad \overrightarrow{v}^j_n(t,x) =
e^{i(t-t^j_n)\Delta}T^j_n\varphi^j,\\
\overrightarrow{u}^{<k}_{(n)}(t,x) =  & \sum_{j<k}
\overrightarrow{u}^j_{(n)}(t,x),  \quad
\overrightarrow{u}^j_{(n)}(t,x)=T^j_n\overrightarrow{U}^j_{\infty}\left(\frac{x-x^j_n}{(h^j_n)^2}\right),\\
&  \big\|\overrightarrow{u}^j_{(n)}(0) -\overrightarrow{v}^j_n(0)\big\|_{L^2}\rightarrow
0.
\end{align*}

By Proposition \ref{decomp:stable:K} and the following observations that
\begin{enumerate}
\item Every solution of \eqref{NLS} in $\KKK^+$ with the
energy less than $E^*_C$, the mass less than $C$ has global finite Strichartz norm by the
definition of $E^*_C$.

\item Proposition \ref{approximation} precludes that all the nonlinear profiles
$\overrightarrow{U}^j_{\infty}$ have finite global Strichartz norm.
\end{enumerate}
we deduce that there is only one profile and
\begin{align*}
E(u^0_{(n)}(0)) \rightarrow E^*_C, \quad u^0_{(n)}(0)\in \KKK^+, \quad
\big\|\widehat{U}^{0}_{\infty}\big\|_{ST^{0}_{\infty}(I)}=\infty, \quad
\big\|w^1_{n}\big\|_{L^{\infty}_t\dot H^1_x}\rightarrow 0.
\end{align*}

If $h^0_{n}\rightarrow 0$, then
$\widehat{U}^0_{\infty}=|\nabla|^{-1} \overrightarrow{U}^0_{\infty}$
solves the $\dot H^1$-critical NLS
\begin{align*}
\left(i\partial_t + \Delta \right)\widehat{U}^0_{\infty} = f_1
(\widehat{U}^0_{\infty})
\end{align*}
and satisfies
\begin{align*}
E^c\left(\widehat{U}^0_{\infty}(\tau^0_{\infty})\right)=E^*_C <m,
\;\; K^c\left(\widehat{U}^0_{\infty}(\tau^0_{\infty})\right)\geq 0,
\;\; \big\|\widehat{U}^0_{\infty}\big\|_{W_1(I)}=\infty.
\end{align*}
However, it is in contradiction with Killip-Visan's result in \cite{KiV:en-NLS:high dim}. Hence $h^0_n\equiv1$,
which implies \eqref{compact:initial}.

Now we show that $\widehat{U}^0_{\infty}=\left<\nabla\right>^{-1}
\overrightarrow{U}^j_{\infty}$ is a global solution, which is the consequence of the compactness of \eqref{compact:initial}.
Suppose not, then we can choose a sequence $t_n \in \R$ which
approaches the maximal existence time. Since
$\widehat{U}^0_{\infty}(t+t_n)$ satisfies the assumption of this
theorem, then applying the above argument to it, we obtain that for some $\psi\in L^2$ and another sequence $(t'_n, x'_n)\in
\R\times \R^d$, as $n\rightarrow +\infty$
\begin{align}\label{compact:t_n}
\left\|\frac{|\nabla|}{\left<\nabla\right>}\left(\overrightarrow{U}^0_{\infty}(t_n)-e^{-it'_n\Delta}\psi(x-x'_n)\right)\right\|_{L^2}\rightarrow
0.
\end{align}
 Let
$
\overrightarrow{v}(t):=e^{it\Delta}\psi.
$
For any $\varepsilon>0$, there exist $\delta>0$ with $I=[-\delta,
\delta]$ such that
\begin{align*}
\big\| \overrightarrow{v}(t-t'_n)\big\|_{V_2(I)}\leq
\varepsilon, \end{align*} which together with \eqref{compact:t_n}
implies that for sufficiently large $n$
\begin{align*}
\big\|
e^{it\Delta}\overrightarrow{U}^0_{\infty}(t_n)\big\|_{V_2(I)} \leq
 \varepsilon.
\end{align*}
If $\varepsilon$ is small enough, this implies that the solution
$\widehat{U}^0_{\infty}$ exists on $[t_n-\delta, t_n+\delta]$ for
large $n$ by the small data theory. This contradicts the choice of
$t_n$. Hence $\widehat{U}^0_{\infty}$ is a global solution and it is just the desired
critical element $u_c$. By Proposition \ref{threshold-energy}, we know that $K(u_c)>0$.
\end{proof}

\subsection{Compactness of the critical element} Since \eqref{NLS} is
symmetric in $t$, we may assume that
\begin{align}\label{critical:infinity Strich}
\big\| u_c\big\|_{ST(0, +\infty)}=\infty.
\end{align}
We call it a forward critical element.

\begin{proposition}\label{compact:APS}For $d\geq 5$.
Let $u_c$ be a forward critical element. Then there exists $x(t):(0,
\infty)\rightarrow \R^d$ such that the set
\begin{align*}
\{u_c(t, x-x(t)); 0<t<\infty\}
\end{align*}
is precompact in $\dot H^s$ for any $s\in(0, 1]$.
\end{proposition}
\begin{proof} By the conservation of the mass, it suffices to prove the precompactness of
$u_c(t_n)\}$ in $\dot H^1$ for any positive time $t_1, t_2, \ldots$.
Suppose $t_n$ converges, then it is trivial from the continuity in
$t$.

Suppose $t_n \rightarrow +\infty$. Applying Theorem
\ref{APS:existence} to the sequence of solutions
$\overrightarrow{u}_c(t+t_n)$, we get another sequence $(t'_n, x'_n)
\in \R \times \R^d$ and $\varphi\in L^2$ such that
\begin{align*}
\frac{|\nabla|}{\left<\nabla\right>}\left(\overrightarrow{u}_c(t_n,x) - e^{-it'_n\Delta}\varphi(x-x'_n)\right)
\rightarrow 0 \quad \text{in}\;\; L^2.
\end{align*}
\begin{enumerate}
\item If $t'_n \rightarrow -\infty$, then we have
\begin{align*}
\big\|\left<\nabla\right>^{-1}e^{it\Delta}\overrightarrow{u}_c(t_n)\big\|_{ST(0,
+\infty)} =
\big\|\left<\nabla\right>^{-1}e^{it\Delta}\varphi\big\|_{ST(-t'_n,
+\infty)} + o_n(1) \rightarrow 0.
\end{align*}
Hence $u_c$ can solve \eqref{NLS}  for $t>t_n$ with large $n$
globally by iteration with small Strichartz norms, which contradicts
\eqref{critical:infinity Strich}.

\item If $t'_n \rightarrow +\infty$, then we have
\begin{align*}
\big\|\left<\nabla\right>^{-1}e^{it\Delta}\overrightarrow{u}_c(t_n)\big\|_{ST(-\infty,
0)} =
\big\|\left<\nabla\right>^{-1}e^{it\Delta}\varphi\big\|_{ST(-\infty,
-t'_n)} + o_n(1) \rightarrow 0
\end{align*}
Hence $u_c$ can solve \eqref{NLS} for $t<t_n$ with large $n$
with diminishing Strichartz norms, which implies $u_c=0$ by taking the
limit, which is a contradiction.
\end{enumerate}

Thus $t'_n$ is bounded, which implies that $t'_n$ is precompact, so is $u_c(t_n,
x+x'_n)$ in $\dot H^1$.
\end{proof}

As a consequence, the energy of $u_c$ stays within a fixed
radius for all positive time, modulo arbitrarily small rest (that is, the spatial scaling function of $u_c$ is $1$). More
precisely, we define the exterior energy by
\begin{align*}
E_{R,c}(u;t)=\int_{|x-c|\geq R} \Big( \big|\nabla u(t,x)\big| ^2   + \big|u(t,x)\big|^{2^*} +
\big|u(t,x)\big|^{\frac{2d+2}{d-1}} \Big) \; dx
\end{align*}
for any $R>0$ and $c\in \R^d$. Then we have
\begin{corollary} \label{compact:almost periodicity} For $d\geq 5$.
Let $u_c$ be a forward critical element. then for any $\varepsilon$,
there exist $R_0(\varepsilon)>0$ and $x(t): (0, +\infty)\rightarrow
\R^d$ such that at any time $t>0$, we have
\begin{align*}
E_{R_0, x(t)}(u_c; t) \leq \varepsilon E(u_c).
\end{align*}
\end{corollary}

\begin{corollary}[\cite{TaoVZ:NLS:mass compact}]\label{Duhamel:Double} For $d\geq 5$.
Let $u_c$ be the critical element as shown in Theorem \ref{APS:existence}. Then for all $t\in \R$,
\begin{align*}
u(t)=& \lim_{T\rightarrow+\infty} i \int^{T}_t e^{i(t-s)\Delta }\Big(f_1(u(s))+f_2(u(s))\Big) ds\\
 =& \lim_{T\rightarrow -\infty} -i \int^{t}_T e^{i(t-s)\Delta }\Big(f_1(u(s))+f_2(u(s))\Big) ds
\end{align*}
as weak limits in $\dot H^s_x$ for any $s\in (0,1]$.
\end{corollary}

\subsection{Zero momentum of the critical element} Next we show that
the critical element can not move with any positive speed in the
sense of energy by it's Galilean invariance.

\begin{proposition}\label{momentum:zero} For $d\geq 5$.
Let $u_c$ be the critical element as shown in Theorem \ref{APS:existence}. then its total momentum, which is
a conserved quantity, vanishes:
\begin{align*}
P(u_c):=2\Im \int_{\R^d}\overline{ u_c(t,x)} \cdot \nabla u_c(t,x) \; dx
\equiv 0.
\end{align*}
\end{proposition}
\begin{proof} We drop the subscript $c$ for simplicity.
Note that the momentum $P(u)$ and the mass $M(u)$ are finite and
conserved. Moreover, $M(u) \not = 0$, otherwise $u$ would be
identically zero and not a critical element.

Let $\widetilde{u}$ be the Galilean boost of $u$ by $\xi_0$, which
is determined later.
\begin{align*}
\widetilde{u}(t,x):=e^{ix\cdot \xi_0}e^{-it|\xi_0|^2}u(t,x-2\xi_0
t),
\end{align*}
then we have
\begin{align*}
\big\|\nabla \widetilde{u}(t)\big\|^2_{L^2} = \big\|\nabla
u(t)\big\|^2_{L^2} + |\xi_0|^2M(u) + \xi_0 \cdot P(u).
\end{align*}
Equivalently, we have $M(\widetilde{u}) = M(u)$ and
\begin{align*}
E(\widetilde{u}) = E(u) + \frac12 |\xi_0|^2M(u) + \frac12 \xi_0
\cdot P(u) , \quad K(\widetilde{u}) = K(u) + 2 |\xi_0|^2M(u) + 2
\xi_0 \cdot P(u).
 \end{align*}

If $P(u) \not =0$, choosing $\xi_0 \in \R^d$ such that
$$ -\frac12 K(u) \leq  |\xi_0|^2M(u) +
\xi_0 \cdot P(u) < 0, $$ then we can find another critical element
$\widetilde{u}$ in $\KKK^+$ with
\begin{align*}
M(\widetilde{u}) = M(u) \leq C ,\quad E(\widetilde{u})<E(u)=E^*_C, \quad
\big\|\widetilde{u}\big\|_{ST(\R)}=+\infty,
\end{align*}
which is in contradiction with the definition of $E^*_C$. Hence
$P(u)\equiv 0$.
\end{proof}

\subsection{Negative regularity} In this subsection, we show that
\begin{proposition}For $d\geq 5$. Let $u_c$ be the critical element as shown in Theorem \ref{APS:existence}, then $u_c\in L^{\infty}\dot H^{-\epsilon}$ for some $\epsilon>0$.
\end{proposition}
\begin{proof}We drop the subscript $c$.  Since we have $u \in L^{\infty}_tH^1$, then we have $u \in L^{\infty}_tL^p_x$, for any $p\in [2, \frac{2d}{d-2}]$.
\begin{align*}
\big\|f_1(u)+f_2(u)\big\|_{L^{\infty}L^r_x}\lesssim \big\|u\big\|_{L^{\infty}_tL^2_x} \big\| u \big\|^{\frac{4}{d-2}}_{L^{\infty}_tL^{p_1}_x}
+ \big\|u\big\|_{L^{\infty}_tL^2_x} \big\| u \big\|^{\frac{4}{d-1}}_{L^{\infty}_tL^{p_2}_x}
\end{align*}
where $2\leq p_1, p_2\leq 2d/(d-2)$ and
\begin{align*}
\frac1r = \frac{1}{2} + \frac{4}{d-2} \times \frac{1}{p_1} = \frac{1}{2} + \frac{4}{d-1} \times \frac{1}{p_2}.
\end{align*}
Therefore for any $r \in [\frac{2(d-1)}{d+3}, \frac{2d}{d+4}]$, we have
\begin{align}\label{Ine:regularity}
f_1(u)+f_2(u) \in L^{\infty}_tL^r_x.
\end{align}

Now from \eqref{Ine:regularity}, we claim that
\begin{align}\label{regularity}
u \in L^{\infty}_t \dot B^{-s_0}_{2, \infty}, \quad \text{for any}\;\; s_0 = \frac{d}{r}-\frac{d+4}{2} \in [0, \frac{2}{d-1}],
\end{align}
which implies the negative regularity of $u$ by interpolation. Now we shows \eqref{regularity}. By the time translation symmetry, it suffices to show that
$u(0) \in \dot B^{-s_0}_{2, \infty}$. In fact, from Corollary \ref{Duhamel:Double},  we have
\begin{align*}
& \big\|u_N(0)\big\|^2_{L^2} \\
=&   \left< i \int^{\infty}_0 e^{-it\Delta}P_N \Big(f_1(u(t))+f_2(u(t))\Big)\; dt, \; -   i \int^0_{-\infty} e^{-i\tau\Delta}P_N \Big(f_1(u(\tau))+f_2(u(\tau)\Big)\; d\tau\right> \\
\leq & \int^{\infty}_0 \int^0_{-\infty} \left| \left<   P_N \Big(f_1(u(t))+f_2(u(t))\Big), \; e^{i(t-\tau)\Delta}P_N \Big(f_1(u(\tau))+f_2(u(\tau)\Big)  \right> \right| \; dt d\tau.
\end{align*}
On one hand, by the dispersive estimate of $e^{it\Delta}$, we have
\begin{align*}
& \left| \left<   P_N \Big(f_1(u(t))+f_2(u(t))\Big), \; e^{i(t-\tau)\Delta}P_N \Big(f_1(u(\tau))+f_2(u(\tau)\Big)  \right> \right|\\
\lesssim & \big\| P_N \Big(f_1(u(t))+f_2(u(t))\Big) \big\|_{L^r_x} \big\| e^{i(t-\tau)\Delta}P_N \Big(f_1(u(\tau))+f_2(u(\tau)\Big)  \big\|_{L^{r'}_x}\\
\lesssim & \big|t-\tau\big|^{d\big(\frac12 - \frac1r\big)}  \big\|  f_1(u(t))+f_2(u(t))  \big\|_{L^r_x}  \big\|  f_1(u(\tau))+f_2(u(\tau)) \big\|_{L^r_x}.
\end{align*}
On the other hand, by Bernstein's inequality, we have
\begin{align*}
& \left| \left<   P_N \Big(f_1(u(t))+f_2(u(t))\Big), \; e^{i(t-\tau)\Delta}P_N \Big(f_1(u(\tau))+f_2(u(\tau)\Big)  \right> \right|\\
\lesssim & \big\| P_N \Big(f_1(u(t))+f_2(u(t))\Big) \big\|_{L^2_x} \big\|  P_N \Big(f_1(u(\tau))+f_2(u(\tau)\Big)  \big\|_{L^{2}_x}\\
\lesssim & N^{2d\big(\frac1r -\frac12\big)}  \big\|  f_1(u(t))+f_2(u(t))  \big\|_{L^r_x}  \big\| f_1(u(\tau))+f_2(u(\tau))  \big\|_{L^r_x}.
\end{align*}
Therefore, for $d\geq 5$, we have
\begin{align*}
  \big\|u_N(0)\big\|^2_{L^2} \lesssim &  \big\| f_1(u)+f_2(u)  \big\|^2_{L^{\infty}_tL^r_x}  \times
  \int^{\infty}_0 \int^0_{-\infty} \min\left(|t-\tau|^{-1}, N^2\right)^{d\left(\frac1r-\frac12\right)} \; dtd\tau \\
\lesssim & N^{2s_0} \big\| f_1(u)+f_2(u)  \big\|^2_{L^{\infty}_tL^r_x},
\end{align*}
where $s_0=\frac{d}{r}-\frac{d+4}{2}.$ This implies \eqref{regularity}.
\end{proof}

\begin{corollary}\label{compact:mass} For $d\geq 5$.
Let $u_c$ be the critical element as shown in Theorem \ref{APS:existence}, then for every $\eta>0$ there exists $C(\eta)>0$ such that
\begin{align*}
\sup_{t\geq 0}\int_{|x-x(t)|\geq C(\eta)} \big|u(t,x ) \big|^2\; dx\lesssim \eta.
\end{align*}
\end{corollary}
\begin{proof}
The proof is the same as Lemma 8.2 in \cite{KiV:en-NLS:high dim}.
\end{proof}

\subsection{Control of the spatial center function of the critical
element}

Now we will use the virial argument to control the spatial center
function of the critical element by acknowledge of the zero momentum and the compactness in $H^1$
of the critical element.

\begin{proposition}\label{control:center function}  For $d\geq 5$.
Let $u_c$ be the critical element as shown in Theorem \ref{APS:existence}. Then
\begin{align*}
|x(t)|=o(t), \quad \text{as}\;\; t\rightarrow +\infty.
\end{align*}
\end{proposition}

\begin{proof}
We argue by contradiction. Suppose that there exist $\delta>0$ and a
sequence $t_n\rightarrow +\infty$ such that
\begin{align*}
|x(t_n)|>\delta t_n \quad \text{for all }\;\; n\geq 1.
\end{align*}

Let $\eta>0$ be a small constant to be chosen later. By compactness
and Corollary \ref{compact:almost periodicity} and Corollary \ref{compact:mass}, there exist $x(t)$ and $C(\eta)$ such that for any $t\geq 0$
\begin{align}\label{compact:apply}
\int_{|x-x(t)|>C(\eta)} \left(|\nabla u (t,x)|^2 +
|u(t,x)|^2\right) \leq \eta.
\end{align}

Define
\begin{align*}
T_n:=\inf \Big\{t \in [0, t_n]\;\big|\; |x(t)|=|x(t_n)|\Big\}\leq t_n, \quad
R_n:=C(\eta)+\sup_{t\in [0,T_n]}|x(t)|.
\end{align*}

Now let $\phi$ be a smooth, radial function satisfying $0\leq \phi \leq 1$, $\phi(x)=1$ for $|x|\leq 1$, and $\phi(x)=0$ for $|x|\geq 2$.
Define the truncation ``position'' as following
\begin{align*}
X_R(t)=\int_{\R^d} x \phi\left(\frac{|x|}{R}\right) \cdot |u(t,x)|^2\; dx,
\end{align*}
then we have
\begin{align*}
\Big|X_{R_n}(0)\Big| \leq  & \left| \int_{|x|\leq C(\eta)}x
\phi\left(\frac{|x|}{R_n}\right) |u(0,x)|^2\; dx \right| + \left|
\int_{|x)|\geq C(\eta)}x \phi\left(\frac{|x|}{R_n}\right)
|u(0,x)|^2\; dx \right|\\
\leq & C(\eta) M(u) + 2\eta R_n,
\end{align*}
and
\begin{align*}
\Big|X_{R_n}(T_n)\Big| \geq  & \big|x(T_n)\big|\cdot M(u) - \big|x(T_n)\big|\cdot \left|\int_{\R^d}
\left(1-\phi\left(\frac{|x|}{R_n}\right)  \right) |u(T_n, x)|^2 \;
dx \right| \\
& - \left| \int_{|x-x(T_n)|\leq C(\eta)}\left(x-x(T_n)\right)
\phi\left(\frac{|x|}{R_n}\right) |u(T_n,x)|^2\; dx \right| \\
& - \left| \int_{|x-x(T_n))|\geq C(\eta)}\left(x-x(T_n)\right)
\phi\left(\frac{|x|}{R_n}\right)
|u(T_n,x)|^2\; dx \right|\\
\geq & \big|x(T_n)\big| \cdot M(u) - \big|x(T_n)\big| \cdot \eta - C(\eta)M(u) - (R_n +
|x(T_n)|)\cdot \eta\\
\geq & \big|x(T_n)\big|\cdot  \left(M(u)-\eta\right) - C(\eta)\cdot (M(u)+\eta).
\end{align*}
Thus, taking $\eta>0$ sufficiently small, we have
\begin{align}\label{position diff:lower bound}
\Big|X_{R_n}(T_n) - X_{R_n}(0)\Big| \geq  \frac{M(u)}{2}  |x(T_n)|  -  2 M(u) \cdot C(\eta).
\end{align}

On the other hand, we have
\begin{align*}
\partial_t X_{R_n}(t) = & 2 \Im \int_{\R^d}
\phi\left(\frac{|x|}{R_n}\right) \nabla u(t,x) \cdot
\overline{u(t,x)} \; dx \\
& + 2 \Im \int_{\R^d} \frac{x}{|x|R_n}
\phi'\left(\frac{|x|}{R_n}\right) x \cdot \nabla u(t,x)
\overline{u(t,x)} \; dx,
\end{align*}
which together with Proposition \ref{momentum:zero} and Corollary
\ref{compact:almost periodicity} implies that for any $t\in [0,
T_n]$
\begin{align*}
\Big|\partial_t X_{R_n}(t)\Big| \leq & \left| 2 \Im \int_{\R^d}
\left(1- \phi\left(\frac{|x|}{R_n}\right)\right) \nabla u(t,x) \cdot
\overline{u(t,x)} \; dx \right| \\
& + \left| 2 \Im \int_{\R^d} \frac{x}{|x|R_n}
\phi'\left(\frac{|x|}{R_n}\right) x \cdot \nabla u(t,x)
\overline{u(t,x)} \; dx \right| \leq C \eta.
\end{align*}

Hence, we have
\begin{align*}
 \frac{M(u)}{2} \cdot \delta T_n  - 2 M(u)\cdot C(\eta) \leq \frac{M(u)}{2} \cdot \delta t_n  - 2 M(u)\cdot C(\eta) < \Big|x(T_n)\Big|  - C(\eta)
\leq C \eta \cdot T_n.
\end{align*}
Taking $\eta$ sufficiently small such that $  C \eta \leq \frac{M(u)}{4} \cdot \delta $, we obtain a contradiction with the fact $T_n \rightarrow +\infty$.
\end{proof}

\subsection{Death of the critical element}
We are in a position to preclude the soliton-like solution by
a truncated Virial identity.
%

\begin{theorem}For $d\geq 5$. The critical element $u_c$ of \eqref{NLS} cannot be a soliton in the sense
of Theorem \ref{APS:existence}.
\end{theorem}
\begin{proof} We still drop the subscript $c$. By Proposition \ref{position diff:lower bound}, for any $\eta>0$, there exists
$T_0=T_0(\eta)\in \R$ such that
\begin{align}
|x(t)|\leq \eta t \;\; \text{for all }\;  t\geq T_0.
\end{align}

Now let $\phi$ be a smooth, radial function satisfying $0\leq \phi \leq 1$, $\phi(x)=1$ for $|x|\leq 1$, and $\phi(x)=0$ for $|x|\geq 2$.
For some $R$, we define
\begin{align*}
V_R(t):=\int_{\R^d} \phi_R(x) |u(t,x)|^2\; dx, \quad
\phi_R(x)=R^2\phi\left(\frac{|x|^2}{R^2}\right).
\end{align*}

On one hand, we have
\begin{align*}
\partial_t V_R(t) = 4 \Im \int_{\R^d}
\phi'\left(\frac{|x|^2}{R^2}\right) x \cdot \nabla u (t,x)\;
\overline{u(t,x)} \; dx.
\end{align*}
Therefore, we have
\begin{align}\label{virial:derivative:upper bound}
\big|\partial_t V_R(t)\big| \lesssim R
\end{align}
for all $t\geq 0$ and $R>0$.

On the other hand, by H\"{o}lder's inequality, we have
\begin{align*}
 & \partial^2_t V_R(t)
=4 \int_{\R^d} \partial_{ij}\big(\phi_R \big) u_i(t,x) \bar{u}_j(t,x) \; dx -\int_{\R^d} (\Delta^2 \phi_R)(x) |u(t,x)|^2 \; dx\\
& \qquad \quad - \frac4d \int_{\R^d} (\Delta \phi_R)(x) |u(t,x)|^{2^*}\; dx + \frac{4}{d+1}
\int_{\R^d} (\Delta \phi_R)(x) |u(t,x)|^{\frac{2d+2}{d-1}}\; dx \\
= & 4 \int_{\R^d} \left( 2 |\nabla u(t,x)|^2 -2|u(t,x)|^{2^*} +\frac{2d}{d+1}
|u(t,x)|^{\frac{2d+2}{d-1}} \right)\; dx \\
+ & O\left(\int_{|x|\geq R} \left(|\nabla u(t,x)|^2  +|u(t,x)|^{2^*} +
|u(t,x)|^{\frac{2d+2}{d-1}} \right) \;dx + \left(\int_{R\leq |x|\leq 2R}
|u(t,x)|^{2^*} \;dx \right)^{(d-2)/d}\right)\\
= & 4 K\left(u(t)\right)
+  O\left(\int_{|x|\geq R} \left(|\nabla u(t,x)|^2  +
|u(t,x)|^{\frac{2d+2}{d-1}}\right) \;dx + \left(\int_{R\leq |x|\leq 2R}
|u(t,x)|^{2^*} \;dx \right)^{(d-2)/d}\right).
\end{align*}

By Lemma \ref{uniform bound}, we
have
\begin{align*}
 4 K\left(u(t)\right)= &\; 4 \int_{\R^d} \left( 2 |\nabla u(t,x)|^2 -2|u(t,x)|^{2^*} +\frac{2d}{d+1}
|u(t,x)|^{\frac{2d+2}{d-1}} \right)\; dx \\
\gtrsim & \min\left(\bar{\mu}(m-E(u(t))), \frac{2}{2d-3} \big\|\nabla u(t)
\big\|^2_{L^2} +  \frac{2d}{(d+1)(2d-3)} \big\|u(t)\big\|^{\frac{2d+2}{d-1}}_{L^{\frac{2d+2}{d-1}}} \right) \\
 \gtrsim &
E(u(t)),
\end{align*}
Thus, choosing $\eta>0$ sufficiently small and
$\displaystyle R:=C(\eta)+\sup_{t\in [T_0, T_1]}|x(t)|$ and by Corollary
\ref{compact:almost periodicity}, we obtain
\begin{align*}
 \partial^2_t V_R(t)  \gtrsim E(u(t)) = E(u_0),
\end{align*}
which implies that for all $T_1>T_0$
\begin{align*} (T_1-T_0)E(u_0) \lesssim
R=C(\eta)+\sup_{t\in [T_0, T_1]}|x(t)| \leq  C(\eta) + \eta T_1.
\end{align*}
Taking $\eta$ sufficiently small and $T_1$
sufficiently large, we obtain a contradiction unless $u\equiv0$. But
$u\equiv 0$ is not consistent with the fact that
$\big\|u\big\|_{ST(\R)}=\infty$.
\end{proof}

\proof[Acknowledgements] The authors are partly supported by the NSF
of China (No. 10901148, No. 11171033, No. 11231006). L. Zhao is
supported by the Fundamental Research Funds for the Central
Universities (WK0010000001, WK00100000013).  The authors would like
to thank Professors J. Holmer, K.~Nakanishi and S. Roudenko for
their valuable communications and comments. \qed

%
%
%
%


\begin{thebibliography}{10}
\newcommand{\msn}[1]{\href{http://www.ams.org/mathscinet-getitem?mr=#1}{\sc MR#1}}

\bibitem{AkaIKN:NLS:combined:blowup}T.~Akahori, S.~Ibrahim, H.~Kikuchi and H.~Nawa, \emph{Existence of a ground state and blow-up problem for a nonlinear
Schr\"{o}dinger equation with critical growth}, to appear in Adv. Diff. Equat..

\bibitem{AkaIKN:NLS:combined:scattering}T.~Akahori, S.~Ibrahim, H.~Kikuchi and H.~Nawa, \emph{Existence of a ground state and
scattering for a nonlinear Schr\"{o}dinger equation with critical growth.} preprint.

\bibitem{Aubin:Sharp contant:Sobolev}T.~Aubin, \emph{Probl\'emes isop\'erim\'etriques et espaces de Sobolev},
J. Diff. Geom., \textbf{11} (1976), 573--598.

\bibitem{BahG:NLW:proffile decomp}H.~Bahouri and P.~G\'{e}rard,
\emph{High frequency approximation of solutions to critical
nonlinear wave equations}, Amer. J. Math., \textbf{121} (1999),
131--175.

\bibitem{Bou:NLS:99}J.~Bourgain, \emph{Global well-posedness of defocusing 3D
critical NLS in the radial case}, J. Amer. Math. Soc.,
\textbf{12}(1999), 145--171.

\bibitem{Bou:NLS:book}J.~Bourgain, \emph{Global solutions of nonlinear Schr\"{o}dinger
equations}, Amer. Math. Soc. Colloq. Publ. \textbf{46}, Amer. Math.
Soc., Providence, 1999.

\bibitem{Caz:NLS:book}T.~Cazenave, \emph{Semilinear Schr\"{o}dinger equations}. Courant
Lecture Notes in Mathematics, Vol. \textbf{10}. New York: New York
University Courant Institute of Mathematical Sciences, 2003.

\bibitem{CKSTT04}J.~Colliander, M.~Keel, G.~Staffilani, H.~Takaoka and T.~Tao, \emph{Global existence and scattering for rough
solution of a  nonlinear Schr\"{o}dinger equation on
$\mathbb{R}^3$}, Comm. Pure Appl. Math., \textbf{57}(2004),
987--1014.

\bibitem{CKSTT07}J.~Colliander, M.~Keel, G.~Staffilani, H.~Takaoka,
and T.~Tao, \emph{Global well-posedness and scattering for the
energy-cirtical nonlinear Schr\"{o}dinger equation in
$\mathbb{R}^3$},  Ann. of Math., \textbf{167}(2008), 767--865.

\bibitem{Dod:NLS:higher dim}B.~Dodson, \emph{Global well-posedness and scattering for the defocusing, $L^2$-critical, nonlienar
Schr\"{o}dinger equation when $d\geq 3$}, J. Amer. Math. Soc.,
\textbf{25:2}(2012),429--463.

\bibitem{Dod:NLS:two dim}B.~Dodson,  \emph{Global well-posedness and scattering for the defocusing, $L^2$-critical, nonlienar
Schr\"{o}dinger equation when $d= 2$}, arXiv:1006.1375.

\bibitem{Dod:NLS:one dim}B.~Dodson,  \emph{Global well-posedness and scattering for the defocusing, $L^2$-critical, nonlienar
Schr\"{o}dinger equation when $d=1$}, arXiv:1010.0040.

\bibitem{Dod:NLS:foc}B.~Dodson, \emph{Global well-posedness and scattering for the mass critical nonlienar
Schr\"{o}dinger equation with mass below the mass of the ground
state}, arXiv:1104.1114.

\bibitem{DuyM:NLS:ThresholdSolution} T.~Duyckaerts and F.~Merle,
\emph{Dynamic of threshold solutions for energy-critical NLS.}
GAFA., \textbf{18:6}(2009), 1787--1840.
%
%

\bibitem{DuyHR:NLS:GWPS}T.~Duyckaerts, J.~Holmer and S.~Roudenko,
\emph{Scattering for the non-radial 3D cubic nonlinear
Schr\"{o}dinger equation}, Math. Res. Lett., \textbf{15:5--6}(2008),
1233--1250.
\bibitem{DuyR:NLS:ThresholdSolution} T.~Duyckaerts and S.~Roudenko, \emph{Threshold solutions for the focusing 3D cubic
Schr\"{o}dinger equation.} Revista. Math. Iber.,
\textbf{26:1}(2010), 1--56.

\bibitem{Fos:NLS:exotic}D.~Foschi, \emph{Inhomogeneous Strichartz estimates}, J. Hyper. Diff. Equat.,
\textbf{2}(2005), 1--24.

\bibitem{GinV:85:NLS}J.~Ginibre and G.~Velo, \emph{Scattering theory in the energy space for a class of nonlinear Schr\"{o}dinger equation},
J. Math. Pures Appl., 64(1985), 363--401.

\bibitem{GinV:85:NLKG subcritical}J.~Ginibre and G.~Velo, \emph{Time decay of finite energy solutions of the nonlinear Klein-Gordon
and Schr\"{o}dinger equation}, Ann. Inst. Henri Poincar\'{e},
Physique th\'{e}orique, \textbf{43:4}(1985), 399--442.

\bibitem{Gla:NLS:blowup}R.~T.~Glassey, \emph{On the blowing up of solutions to the Cauchy problem for nonlinear Schr\"{o}dinger equations},
J. Math. Phys.,
\textbf{18:9}(1977), 1794--1797.

\bibitem{HolRou:NLS:GWPS}J.~Holmer and S.~Roudenko, \emph{A sharp condition for scattering of the radial 3d cubic nonlinear Schr\"{o}dinger
equation}, Comm. Math. Phys., \textbf{282:2}(2008), 435--467.

\bibitem{HolRou:NLS:BU}J.~Holmer and S.~Roudenko, \emph{Divergence to infinite-variance nonradial solutions to the 3d NLS
equation},  Comm. PDEs, \textbf{35:5}(2010), 878--905.

\bibitem{IbrMN:f:NLKG}S.~Ibrahim, N.~Masmoudi and K.~Nakanishi,
\emph{Scattering threshold for the focusing nonlinear Klein-Gordon
equation}, Analysis $\&$ PDE, \textbf{4:3}(2011),405--460

\bibitem{KeT98}M. Keel and T. Tao, \emph{Endpoint Strichartz estimates}, Amer. J.
Math. \textbf{120:5}(1998), 955--980.

\bibitem{KenM:NLS:GWP}
C.~E.~Kenig and F.~Merle, \emph{Global well-posedness, scattering
and blow up for the energy-critical, focusing, nonlinear
Schr\"odinger equation in the radial case}, Invent. Math.,
\textbf{166}(2006), 645--675.

\bibitem{KenM:NLW:GWP}
C.~E.~Kenig and F.~Merle, \emph{Global well-posedness, scattering
and blow-up for the energy critical focusing non-linear wave
equation}, Acta Math., \textbf{201:2}(2008), 147--212.

\bibitem{Ker:NLS:profile decomp}S.~Keraani, \emph{On the defect of compactness for the
Strichartz estimates of the Schr\"{o}dinger equations.} J. Diff.
Equat., \textbf{175}(2001) 353--392.

\bibitem{Ker:NLS:compactness}S.~Keraani, \emph{On the blow up phenomenon of the critical
Schr\"{o}dinger equation}, J. Funct. Anal., \textbf{265}(2006),
171--192.

\bibitem{KiTV:NLS:2d}R.~Killip, T.~Tao and M.~Visan, \emph{The cubic nonlinear
Schr\"{o}dinger equation in two dimensions with radial data}, J. Eur. Math. Soc., \textbf{11:6}(2009), 1203--1258.

\bibitem{KiV:en-NLS:high dim}R.~Killip and M.~Visan, \emph{The focusing energy-critical
nonlinear Schr\"{o}dinger equation in dimensions five and higher},
Amer. J. Math., \textbf{132:2}(2010), 361--424.

\bibitem{Kiv:Clay Lecture}R.~Killip and M.~Visan, \emph{Nonlinear Schr\"{o}dinger euqations at critical regularity}.
Proceedings of the Clay summer school ``Evolution Equations'', June 23-July  18, 2008, Eidgen\"{o}ssische
technische Hochschule, Z\"{u}rich.

\bibitem{KiVZ:NLS:high dim}R.~Killip, M.~Visan and X.~Zhang, \emph{The mass-critical
nonlinear Schr\"{o}dinger equation with radial data in dimensions
three and higher}, Analysis $\&$ PDE. \textbf{1:2}(2008), 229--266.


\bibitem{KriNS:e-critical NLW}J.~Krieger, K.~Nakanishi and
W.~Schlag, \emph{Global dynamics away from the ground state for the
energy-critical nonlinear wave equation}, to appear in Amer. J.
Math..

\bibitem{KriNakSch:1D KG}J.~Krieger, K.~Nakanishi and W.~Schlag,
\emph{Global dynamics above the ground state energy for the
one-dimensional NLKG equation}, Math. Z., \textbf{272:1-2}(2012),
297--316.

\bibitem{LiZh:NLS}D.~Li and X.~Zhang, \emph{Dynamics for the energy critical nonlinear Schr\"{o}dinger equation in high
dimensions}, J. Funct. Anal., \textbf{256:6}(2009), 1928--1961.

\bibitem{MiaoWX:Har:dynamic}C.~Miao, Y.~Wu and G.~Xu, \emph{Dynamics for the focusing, energy-critical nonlinear Hartree
equation}, Forum Math., DOI:10.1515/Forum-2011-0087.

\bibitem{MiXZ:Hartree:De-rad}C.~Miao , G.~Xu and L.~Zhao,
\emph{Global well-posedness and scattering for the energy-critical,
defocusing Hartree equation for radial data}, J. Funct. Anal.,
\textbf{253:2}(2007), 605--627.

\bibitem{MiXZ07b}C.~Miao , G.~Xu and L.~Zhao, \emph{Global well-posedness, scattering and blow-up
for the energy-critical, focusing Hartree equation in the radial
case}, Colloquium Mathematicum, \textbf{114}(2009), 213--236.

\bibitem{MiXZ:Hartree:low regul}C.~Miao , G.~Xu and L.~Zhao,
\emph{Global well-posedness and uniform bound for the defocusing
$H^{1/2}$-subcritical Hartree equation in $\R^d$}, Ann I. H.
Poincar\'{e}, AN., \textbf{26:5}(2009), 1831--1852.

\bibitem{MiXZ:Hartree:mass critical}C.~Miao , G.~Xu and L.~Zhao, \emph{Global well-posedness and scattering for the
mass-critical Hartree equation with radial data}, J. Math. Pures
Appl., \textbf{91}(2009), 49--79.

\bibitem{MiXZ:NLS:e-crit Hartree:nonradial}C.~Miao , G.~Xu and L.~Zhao, \emph{Global well-posedness and scattering for the energy-critical,
defocusing Hartree equation in $\R^{1+n}$}, Comm. PDEs.
\textbf{36:5}(2011), 729--776.

\bibitem{MiaoWZ:NLS:3d Combined} C.~Miao , G.~Xu and L.~Zhao, \emph{The dynamics of the 3D radial NLS with the combined terms},
Comm. Math. Phys., \textbf{318:3}(2013), 767--808.

\bibitem{Nak:NLKG low dim:subcrit}K.~Nakanishi, \emph{Energy scattering for nonlinear Klein-Gordon and Schr\"{o}dinger equations in spatial
dimensions $1$ and $2$}, J. Funct. Anal., \textbf{169}(1999),
201--225.

\bibitem{Nak:01:NLKG subcritical}K.~Nakanishi, \emph{Remarks on the energy scattering for nonlinear Klein-Gordon and Schr\"{o}dinger
equations}, Tohoku Math. J., \textbf{53}(2001), 285--303.



\bibitem{NakSch:cubic NLS:Rigidity}K.~Nakanishi and W.~Schlag,
\emph{Global dynamics above the ground state energy for the cubic
NLS equation in $3D$}, Calc. of Variations and PDE,
\textbf{44:1-2}(2012), 1--45.


\bibitem{OgaTsu:Blowup:NLS:91}T.~Ogawa and Y.~Tsutsumi, \emph{Blow-up of $H^1$ solution for the nonlinear Schr\"{o}dinger equation}, J. Diff. Equat.,
\textbf{92}(1991), 317--330.

\bibitem{RyV05}E.~Ryckman and M.~Visan, \emph{Global well-posedness and
scattering for the defocusing energy-critical nonlinear
Schr\"{o}dinger equation in $\mathbb{R}^{1+4}$}, Amer. J. Math.,
\textbf{129}(2007), 1--60.

\bibitem{Talenti:best constant}G.~Talenti, \emph{Best constant in Sobolev inequality}, Ann. Mat. Pura. Appl. \textbf{110}(1976), 353--372.

\bibitem{tao:book}T.~Tao, \emph{Nonlinear dispersive equations, local and global
analysis}, CBMS. Regional conference Series in Mathematics,
\textbf{106}. Published for the Conference Board of the Mathematical
Science, Washington, DC; by the American Mathematical Society
Providence, RI, 2006.

\bibitem{TaoVZ:NLS:combined}T.~Tao, M.~Visan and X.~Zhang, \emph{The nonlinear Schr\"{o}dinger equation with combined power-type
nonlinearities}, Comm. PDEs, \textbf{32}(2007), 1281--1343.

\bibitem{TaoVZ:NLS:mass compact}T.~Tao, M.~Visan and X.~Zhang, \emph{Minimal-mass blowup solutions of the mass-critical NLS},
Forum Math., \textbf{20:5}(2008), 881--919.

\bibitem{Vi05}M.~Visan, \emph{The defocusing energy-critical nonlinear
Schr\"{o}dinger equation in higher dimensions}, Duke Math. J.,
\textbf{138:2}(2007), 281--374.

\bibitem{Zha:global:NLKG}J.~Zhang, \emph{Sharp conditions of global existence for nonlinear Schr\"{o}dinger and
Klein-Gordon equations}, Nonlinear Anal. T. M. A.,
\textbf{48:2}(2002), 191--207.

\bibitem{Zhang:NLS:06}X.~Zhang,
\emph{On Cauchy problem of 3D energy critical Schr\"{o}dinger
equation with subcritical perturbations}, J. Diff. Equat.,
\textbf{230:2}(2006), 422--445.
\end{thebibliography}
\end{document}